\documentclass[10pt]{article}
\usepackage{amsmath, amsthm, amssymb, amsfonts, mathrsfs, mathtools}
\usepackage{graphicx}
\usepackage{makeidx}
\usepackage[left=2cm,right=2cm,top=2cm,bottom=2cm]{geometry}
\usepackage{caption}
\usepackage{subcaption}
\usepackage[section]{placeins}
\usepackage{enumitem}
\usepackage{tikz}
\usepackage{stmaryrd}

\newcommand{\Mod}[1]{\ (\mathrm{mod}\ #1)}

\usetikzlibrary{decorations.pathreplacing}
\tikzstyle{vertex}=[auto=left,circle,draw=black,fill=white, inner sep=1.5]

\usepackage{hyperref}
\hypersetup{colorlinks=true, citecolor={blue}, linkcolor=blue, filecolor=magenta, urlcolor=cyan}

\newtheorem{theorem}{Theorem}[section]

\newtheorem{lema}[theorem]{Lemma}
\newtheorem{corollary}{Corollary}[theorem]

\newtheorem{ex}{Example}[section] 

\renewcommand{\i}{\mathbf{i}}

\newcommand{\Cay}{\operatorname{Cay}}

\newcommand{\ord}{\operatorname{ord}}
\newcommand{\Cl}{\operatorname{Cl}}
\newcommand{\Tr}{\operatorname{Tr}}
\newcommand{\Gal}{\operatorname{Gal}}
\newcommand{\IRR}{\operatorname{IRR}}
\newcommand{\Irr}{\operatorname{Irr}}

\title{HS-integral and Eisenstein integral normal mixed Cayley graphs}

\author{Monu Kadyan\\
Department of Mathematics\\
Indian Institute of Technology Guwahati, India\\
monu.kadyan@iitg.ac.in}

\date{}

\begin{document}
	\maketitle
	
	\vspace{-0.3in}
	
\begin{center}{\textbf{Abstract}}\end{center}
	
\noindent A mixed graph is said to be HS-\emph{integral} if the eigenvalues of its Hermitian-adjacency matrix of the second kind are integers. A mixed graph is called \emph{Eisenstein integral} if the eigenvalues of its (0, 1)-adjacency matrix are Eisenstein integers. We characterize the set $S$ for which the normal mixed Cayley graph $\text{Cay}(\Gamma, S)$ is HS-integral for any finite group $\Gamma$. We further show that a normal mixed Cayley graph is HS-integral if and only if it is Eisenstein integral. This paper generalizes the results of [M. Kadyan, B. Bhattacharjya. HS-integral and Eisenstein integral mixed Cayley graphs over abelian groups. Linear Algebra Appl. 645:68-90, 2022].

\vspace*{0.3cm}

\section{Introduction}

A \emph{mixed graph} $G$ is a pair $(V(G),E(G))$, where $V(G)$ and $E(G)$ are the vertex and edge sets of $G$, respectively. Here $E(G) \subseteq V(G) \times V(G)\setminus \{(u,u):u\in V(G)\}$. If $G$ is a mixed graph, then $(u,v)\in E(G)$ need not imply that $(v,u)\in E(G)$; see \cite{2015mixed} for further information. If both $(u,v)$ and $(v,u)$ are members of $E(G)$, then $(u,v)$ is referred to as an \textit{undirected edge}. If only one of $(u,v)$ and $(v,u)$ is a member of $E(G)$, then it is called a \textit{directed edge}. As a result, both undirected and directed edges can exist simultaneously in a mixed graph. If all of the edges of $G$ are undirected (resp. directed), we refer to $G$ as a \textit{simple graph} (resp. an \textit{oriented graph}). Some definitions and results of this paper have similarities with those in the paper~\cite{kadyanArxivNorm}. Throughout the paper, we consider $\mathbf i = \sqrt{-1}$ and $\omega_n :=\exp{(\frac{2\pi \mathbf i}{n})}$.  

Assume that $G$ is a mixed graph with $n$ vertices. The (0,1)-\textit{adjacency matrix} and the \textit{Hermitian-adjacency matrix of the second kind} of $G$ are denoted by $\mathcal{A}(G)=(a_{uv})_{n\times n}$ and $\mathcal{H}(G)=(h_{uv})_{n\times n}$, respectively, where 
\[a_{uv} = \left\{ \begin{array}{rl}
	1 &\mbox{ if }
	(u,v)\in E \\ 
	0 &\textnormal{ otherwise,}
\end{array}\right.     ~~~~~\text{ and }~~~~~~ h_{uv} = \left\{ \begin{array}{cl}
	1 &\mbox{ if }
	(u,v)\in E \textnormal{ and } (v,u)\in E \\ \frac{1+\i\sqrt{3}}{2} & \mbox{ if } (u,v)\in E \textnormal{ and } (v,u)\not\in E \\
	\frac{1-\i\sqrt{3}}{2} & \mbox{ if } (u,v)\not\in E \textnormal{ and } (v,u)\in E\\
	0 &\textnormal{ otherwise.}
\end{array}\right.\] 

The Hermitian-adjacency matrix of the second kind was presented by Bojan Mohar \cite{mohar2020new}. An eigenvalue of $\mathcal{H}(G)$ is referred to an \emph{HS-eigenvalue} of $G$. An eigenvalue of $\mathcal{A}(G)$ is known as an \emph{eigenvalue} of $G$. Similarly, the \emph{HS-spectrum} of $G$ is the multi-set of the HS-eigenvalues of $G$, and the \emph{spectrum} of $G$ is the multi-set of the eigenvalues of $G$. The Hermitian-adjacency matrix of the second kind of a mixed graph is a Hermitian matrix, so its HS-eigenvalues are real numbers. However, if a mixed graph $G$ has at least one directed edge, then $\mathcal{A}(G)$ is not a Hermitian matrix (or symmetric). As a result, the eigenvalues of $G$ need not be real numbers. 

A mixed graph $G$ is said to be \textit{HS-integral} if all of its HS-eigenvalues are integers. A mixed graph $G$ is said to be \textit{Eisenstein integral} if all of its eigenvalues are Eisenstein integers. Note that complex numbers of the form $a+b\omega_3$, where $a,b\in \mathbb{Z}$, are known as \emph{Eisenstein integers}. Note that $\mathcal{A}(G)=\mathcal{H}(G)$ for a simple graph $G$. Therefore, the term \emph{integral graph} refers to an HS-integral simple graph. As a result, the words HS-eigenvalue, HS-spectrum and HS-integrality of a simple graph $G$ have the same meaning with that of eigenvalue, spectrum and integrality of $G$, respectively.

In 1974, Harary and Schwenk~\cite{harary1974graphs} raised the question of characterization of integral  graphs. This problem has inspired a lot of interest over the last half-century. For more information on integral graphs, we refer the reader to \cite{ahmadi2009graphs, balinska2002survey, csikvari2010integral, watanabe1979note, watanabe1979integral}.

Throughout the paper, we consider $\Gamma$ to be a finite group and ${\mathbf 1}$ to be the identity element of $\Gamma$. Let $S$ be a subset of $\Gamma$ that does not contain the identity element, that is, ${\mathbf 1} \not \in S$.  If $S$ is closed under inverse (resp. $a^{-1} \not\in S$ for all $a\in S$), it is said to be \textit{symmetric} (resp. \textit{skew-symmetric}). Define $\overline{S}= \{u\in S: u^{-1}\not\in S \}$. Then $S\setminus \overline{S}$ is symmetric, while $\overline{S}$ is skew-symmetric. The \textit{mixed Cayley graph} $G=Cay(\Gamma,S)$ is a mixed graph with $V(G)=\Gamma$ and $E(G)=\{ (a,b):a,b\in \Gamma, ba^{-1}\in S \}$. If $S$ is symmetric (resp. skew-symmetric), we refer $G$ to be a \textit{simple Cayley graph} (resp. \textit{oriented Cayley graph}). A mixed Cayley graph $Cay(\Gamma, S)$ is called \textit{normal} if $S$ is the union of some conjugacy classes of the group $\Gamma$.

In 1982, Bridge and Mena \cite{bridges1982rational} presented a characterization of integral Cayley graphs over abelian groups. Later on, same characterization was obtained by \cite{alperin2012integral,klotz2010integral,2006integral}. For results on integral Cayley graphs over non-abelian groups, we recommend the reader to \cite{cheng2019integral,ku2015cayley, lu2018integral}.   The HS-integrality and Eisenstein integrality of mixed Cayley graphs over abelian groups and cyclic groups are characterized in~\cite{kadyan2021hsIntegralAbelian} and \cite{kadyan2021Secintegral}, respectively. In 2014, Godsil \emph{et al.} \cite{godsil2014rationality} characterized integral normal Cayley graphs. 


The paper is organized as follows. In Section 2, we present some preliminary notions and known results. We also express the HS-eigenvalues of a normal mixed Cayley graph $\Cay(\Gamma, S)$ in terms of the irreducible characters of $\Gamma$. In section 3, we find a characterization of HS-integral normal oriented Cayley graphs. In section 4, we extend the characterization obtained in Section 3 to normal mixed Cayley graphs. In the last section, we show that a normal mixed Cayley graph is HS-integral if and only if it is Eisenstein integral. 


\section{Preliminaries}


For $x\in \Gamma$, let $\ord(x)$ denote the order of $x$. If $g$ and $h$ are elements of the group $\Gamma$, then we call $h$ a \textit{conjugate}\index{conjugate} of $g$ if $g=x^{-1}hx$ for some $x\in \Gamma$. The \textit{conjugacy class}\index{conjugacy class} of $g$, denoted ${\rm Cl}(g)$, is the set of all conjugates of $g$ in $\Gamma$. Define $C_{\Gamma}(g)$ to be the set of all elements of $\Gamma$ that commute with $g$. We denote the \textit{group algebra}\index{group algebra} of $\Gamma$ over a field $\mathbb{F}$ by $\mathbb{F}\Gamma$. That is, $\mathbb{F}\Gamma$ is the set of all formal sums $\sum\limits_{g\in \Gamma}a_g g$, where $a_g\in \mathbb{F}$, and we assume $1.g=g$ to have $\Gamma \subseteq \mathbb{F}\Gamma$.

A \textit{representation}\index{representation} of a finite group $\Gamma$ is a homomorphism $\rho \colon \Gamma \to \text{GL}_n(\mathbb{C})$, where $\text{GL}_n(\mathbb{C})$ is  the set of all $n\times n$ invertible matrices with complex entries. Here, the number $n$ is called the \textit{degree}\index{degree} of $\rho$. Two representations $\rho_1$ and $\rho_2$ of $\Gamma$ of degree $n$ are \textit{equivalent}\index{equivalent} if there is a $T\in \text{GL}_n(\mathbb{C})$ such that $T\rho_1(x)=\rho_2(x)T$ for each $x\in \Gamma$.

Let $\rho \colon \Gamma \to \text{GL}_n(\mathbb{C})$ be a representation of $\Gamma$. The \textit{character}\index{character} $\chi_{\rho}\colon \Gamma \to \mathbb{C}$ of $\rho$ is defined by setting $\chi_{\rho}(x):=\Tr(\rho(x))$ for $x\in \Gamma$, where $\Tr(\rho(x))$ is the trace of $\rho(x)$. By degree of $\chi_{\rho}$, we mean the degree of $\rho$, which is simply $\chi_{\rho}(\textbf{1})$. If $W$ is a $\rho(x)$-invariant subspace of $\mathbb{C}^n$ for each $x\in \Gamma$, then we say that $W$ is a $\rho(\Gamma)$-invariant subspace of $\mathbb{C}^n$. If $\{ \mathbf{0}\}$ and $\mathbb{C}^n$ are the only $\rho(\Gamma)$-invariant subspaces of $\mathbb{C}^n$, then  we say $\rho$ an \textit{irreducible representation}\index{irreducible representation} of $\Gamma$, and the corresponding character $\chi_{\rho}$ an \textit{irreducible character}\index{irreducible character} of $\Gamma$. 

For a group $\Gamma$, we denote by $\IRR(\Gamma)$ and $\Irr(\Gamma)$ the complete set of non-equivalent irreducible representations of $\Gamma$  and the complete set of non-equivalent irreducible characters of $\Gamma$, respectively. For $z\in \mathbb{C}$, let $\overline{z}$ denote the complex conjugate of $z$ and $\Re (z)$ (resp. $\Im (z)$) denote the real part (resp. imaginary part) of the complex number $z$.

\begin{theorem}[\cite{steinberg2009representation}]\label{NewThmChap1Added}
Let $\Gamma$ be a finite group and $\rho$ be a representation of $\Gamma$ of degree $k$ with corresponding  character $\chi$. If $x\in \Gamma$ and $\ord(x)=m$, then the following assertions hold.
\begin{enumerate}[label=(\roman*)]
\item $\rho(x)$ is similar to a diagonal matrix with diagonal entries $\epsilon_1,\hdots,\epsilon_k$, where $\epsilon_i^m=1$ for each \linebreak[4] $i\in \{1,\hdots,k\}$.
\item $\chi(x)= \sum\limits_{i=1}^{k}\epsilon_i$, where $\epsilon_i^m=1$ for each $i\in \{1,\hdots,k\}$.
\item $\chi(x^{-1})=\overline{\chi(x)}$.
\end{enumerate}
\end{theorem} 
\begin{proof} Note that $\rho(x)^m$ is an identity matrix. Therefore, $\rho(x)$ is diagonalizable, and that its eigenvalues  are $m$-th roots of unity. Thus the proofs of  Part (i) and Part (ii) follow. 

Again, $xx^{-1}=\mathbf{1}$ gives that $\rho(x^{-1})=\rho(x)^{-1}$. Therefore if $\chi(x)= \sum_{i=1}^{k}\epsilon_i$, then we have that $\chi(x^{-1})= \sum_{i=1}^{k}\epsilon_i^{-1}= \sum_{i=1}^{k}\overline{\epsilon}_i= \overline{\chi(x)}$.
\end{proof}

For a representation $\rho\colon \Gamma \to {\rm GL}_n(\mathbb{C})$ of $\Gamma$, define $\overline{\rho}\colon \Gamma \to {\rm GL}_n(\mathbb{C})$ by $\overline{\rho}(x):=\overline{\rho(x)}$, where $\overline{\rho(x)}$ is the matrix whose entries are the complex conjugates of the corresponding entries of $\rho(x)$. Note that if $\rho$ is irreducible, then $\overline{\rho}$ is also irreducible. Hence we have the following lemma. See Proposition 9.1.1 and Corollary 9.1.2 in \cite{steinberg2009representation} for details. 

\begin{lema}[\cite{steinberg2009representation}]\label{newLemmaConjugateChara}
	Let $\Gamma$ be a finite group and ${\rm Irr}(\Gamma)=\{ \chi_{1},\ldots,\chi_h \}$. If $j\in \{ 1,\ldots,h\}$, then there exists $k\in \{ 1,\ldots,h\}$ satisfying $\overline{\chi}_k=\chi_j$, where $\overline{\chi}_k\colon \Gamma \to \mathbb{C}$ such that $\overline{\chi}_k(x)=\overline{\chi_k(x)}$ for each $x \in \Gamma$.
\end{lema}

\begin{theorem}[\cite{steinberg2009representation}]\label{OrthoCharaRepr}
Let $\Gamma$ be a finite group and $x,y \in \Gamma$. If $\Irr(\Gamma)=\{ \chi_1,\ldots,\chi_h \}$, then 
\begin{enumerate}[label=(\roman*)]
\item 
\begin{align}
	\sum_{x\in \Gamma} \chi_{j}(x) \overline{\chi_{k} (x)} = \left\{ \begin{array}{cl}
		|\Gamma| & \mbox{ if $j=k$} \\
		0 &   \mbox{ otherwise}, 
	\end{array}\right.\nonumber
\end{align}
\item 
\begin{align}
	\sum_{j=1}^{h} \chi_{j}(x) \overline{\chi_{j} (y)} = \left\{ \begin{array}{ll}
		|C_{\Gamma}(x)| & \mbox{ if $x$ and $y$ are conjugates to each other} \\
		0 &   \mbox{ otherwise}. 
	\end{array}\right.\nonumber
	\end{align}
\end{enumerate}
\end{theorem}

For a function $f \colon \Gamma \to \mathbb{C}$, let $[f(yx^{-1})]_{x,y \in \Gamma}$ be the matrix whose rows and columns are indexed by the elements of $\Gamma$, and for $x,y \in \Gamma$, the $(x,y)$-th entry of the matrix is $f(yx^{-1})$. 

\begin{theorem}[\cite{foster2016spectra}]\label{EigNorColCayMix1}
Let $\Gamma$ be a finite group and $\Irr(\Gamma)=\{\chi_1,\ldots,\chi_h \}$. If $f \colon \Gamma \to \mathbb{C}$ is a class function, then the spectrum of the matrix $[f(yx^{-1})]_{x,y \in \Gamma}$ is $\{[\gamma_1]^{d_1^2},\ldots,[\gamma_h]^{d_h^2}\}$, where $$\gamma_{j} = \frac{1}{\chi_j(\mathbf 1)} \sum_{x\in \Gamma} f(x)\chi_{j}(x) \hspace{0.2cm} \textnormal{ and } \hspace{0.2cm} d_j = \chi_j(\mathbf{1})$$ for each $j \in \{ 1,\ldots , h\}$.
\end{theorem} 

\begin{lema}\label{EigNorCayMix}
Let $\Gamma$ be a finite group. If $\Irr(\Gamma)=\{ \chi_1,\ldots,\chi_h \}$, then the HS-spectrum of the normal mixed Cayley graph $\Cay(\Gamma, S)$ is $\{ [\gamma_{1}]^{d_1^2},\ldots, [\gamma_{h}]^{d_h^2} \},$ where $\gamma_j=\lambda_j + \mu_j$, $$\lambda_j= \frac{1}{\chi_j(\mathbf 1)} \sum_{s \in S \setminus \overline{S}} \chi_j(s),\hspace{0.2cm} \mu_j=\frac{1}{\chi_j(\mathbf 1)}  \sum_{s \in \overline{S}}(\omega_6 \chi_j(s)+\omega_6^5 \chi_j(s^{-1})),$$ and $ d_j=\chi_j(\mathbf 1) \textnormal{ for each } j\in \{1,\ldots,h\}.$
\end{lema}
\begin{proof}
Let $f\colon \Gamma \to \{0,1,\omega_6,\omega_6^5\}$ be defined by
	$$f(s)= \left\{ \begin{array}{rl}
		1 & \mbox{if } s\in S \setminus \overline{S} \\
		\omega_6 & \mbox{if } s\in \overline{S}\\
		\omega_6^5 & \mbox{if } s\in \overline{S}^{-1}\\ 
		0 &   \mbox{otherwise}. 
	\end{array}\right.
	$$
Since $S$ is a union of some conjugacy classes of $\Gamma$, $f$ is a class function. The Hermitian adjacency matrix of the second kind of $\Cay(\Gamma, S)$ is given by $[f(yx^{-1})]_{x,y \in \Gamma}$. By Theorem~\ref{EigNorColCayMix1}, $$\gamma_j= \frac{1}{\chi_j(\mathbf 1)} \bigg( \sum_{s \in S \setminus \overline{S}} \chi_j(s) + \sum_{s \in \overline{S}}\omega_6 \chi_j(s)+ \sum_{s \in \overline{S}^{-1}} \omega_6^5 \chi_j(s) \bigg),$$ and the result follows.
\end{proof}

As special cases of Lemma~\ref{EigNorCayMix}, we have the following two corollaries.

\begin{corollary}\label{coro1EigNorCayMix}
Let $\Gamma$ be a finite group. If $\Irr(\Gamma)=\{ \chi_1,\ldots,\chi_h \}$, then the HS-spectrum (or spectrum) of the normal simple Cayley graph $\Cay(\Gamma, S)$ is $\{ [\lambda_{1}]^{d_1^2},\ldots, [\lambda_{h}]^{d_h^2} \},$ where $$\lambda_j= \frac{1}{\chi_j(\mathbf 1)} \sum_{s \in S } \chi_j(s) \textnormal{ and } d_j=\chi_j(\mathbf 1) \hspace{0.2cm} \textnormal{ for each } j\in \{1,\ldots,h\}.$$
\end{corollary}

\begin{corollary}\label{coro2EigNorCayMix}
Let $\Gamma$ be a finite group. If $\Irr(\Gamma)=\{ \chi_1,\ldots,\chi_h \}$, then the HS-spectrum of the normal oriented Cayley graph $\Cay(\Gamma, S)$ is $\{ [\mu_{1}]^{d_1^2},\ldots, [\mu_{h}]^{d_h^2} \},$ where $$ \mu_j=\frac{1}{\chi_j(\mathbf 1)}  \sum_{s \in S}( \omega_6\chi_j(s)+\omega_6^5\chi_j(s^{-1})) \textnormal{ and } d_j=\chi_j(\mathbf 1) \hspace{0.2cm} \textnormal{ for each } j\in \{1,\ldots,h\}.$$
\end{corollary}

Let $n\geq 2$ be a positive integer.  For a divisor $d$ of $n$, define $G_n(d)=\{k: 1\leq k\leq n-1, \gcd(k,n)=d \}$. It is clear that $G_n(d)=dG_{\frac{n}{d}}(1)$.

Let $\mathbb{B}({\Gamma})$ be the boolean algebra\index{boolean algebra} generated by the subgroups of $\Gamma$. That is, $\mathbb{B}({\Gamma})$ is the set whose elements are obtained by intersections, unions and complements of subgroups of $\Gamma$. Define an equivalence relation $\sim$ on $\Gamma$ such that $x\sim y$ if and only if $y=x^k$ for some $k\in G_m(1)$, where $m=\ord(x)$. For $x\in \Gamma$, let $[ x ]$ denote the equivalence class of $x$ with respect to the relation $\sim$. Note that minimal non-empty sets in a boolean algebra are called its \textit{atoms}\index{atoms}. 

\begin{theorem}[\cite{alperin2012integral}]\label{NewThmAtomsAndEquiv}
The atoms of the boolean algebra $\mathbb{B}(\Gamma)$ are the sets $[x]$ for each $x\in \Gamma$.
\end{theorem}

By Theorem~\ref{NewThmAtomsAndEquiv}, we observe that each element of $\mathbb{B}(\Gamma)$ can be expressed as a disjoint union of the equivalence classes of the relation $\sim$ on $\Gamma$. Thus
\[\mathbb{B}(\Gamma)=\{[ x_1 ]\cup\cdots\cup [ x_k ]\colon  x_1,\ldots,x_k\in \Gamma, k\in \mathbb{N}\}.\]

\begin{theorem}[\cite{godsil2014rationality}]\label{NorMixCayGraphInteg}
Let $\Gamma$ be a finite group and ${\rm Cay}(\Gamma, S)$ be a normal simple Cayley graph. Then  ${\rm Cay}(\Gamma, S)$ is integral if and only if $S \in \mathbb{B}(\Gamma)$.
\end{theorem}

Let $n\equiv 0 \pmod 3$. For a divisor $d$ of $\frac{n}{3}$ and $r\in \{ 1,2\}$, define 
$$G_{n,3}^r(d)=\{ dk: k\equiv r \Mod 3, \gcd(dk,n )= d \}.$$ 
It is easy to see that $G_{n}(d)=G_{n,3}^1(d) \cup G_{n,3}^2(d)$ is a disjoint union and $G_{n,3}^r(d)=dG_{\frac{n}{d}}^r(1)$ for $r=1,2$.

Let $\Gamma(3)$ be the set of all $x\in \Gamma$ satisfying $\ord(x)\equiv 0 \pmod 3$. That is, $\Gamma(3):=\{x\in \Gamma\colon \ord(x)\equiv 0 \pmod 3\}$. Define an equivalence relation $\simeq$ on $\Gamma(3)$ such that $x \simeq y$ if and only if $y=x^k$ for some  $k\in G_{m,3}^1(1)$, where $m=\ord(x)$. Observe that if  $x,y\in \Gamma(3)$ and $x \simeq y$ then $x \sim y$, but the converse need not be true. For example, consider $x=5\pmod {12}$, $y=7\pmod {12}$ in $\mathbb{Z}_{12}$. Here $x,y\in \mathbb{Z}_{12}(3)$ and $x \sim y$, but $x \not\simeq y$. For $x\in \Gamma(3)$, we denote the equivalence class of $x$ with respect to the relation $\simeq$ by $\langle\!\langle x \rangle\!\rangle$. For $\Gamma(3) \neq \emptyset$, define $\mathbb{E}(\Gamma)$ to be the set of all skew-symmetric subsets $S$, where $S=\langle\!\langle x_1 \rangle\!\rangle\cup \cdots \cup \langle\!\langle x_k \rangle\!\rangle$ for some $x_1,\ldots,x_k\in \Gamma(3)$. For  $\Gamma(3) = \emptyset$, define $\mathbb{E}(\Gamma):=\{ \emptyset \}$. Thus
\[\mathbb{E}(\Gamma)=\left\{ \begin{array}{ll}
	\{\langle\!\langle x_1 \rangle\!\rangle\cup \cdots \cup \langle\!\langle x_k \rangle\!\rangle \colon x_1,\ldots,x_k\in \Gamma(3), k\in \mathbb{N}\} &\mbox{ if }\Gamma(3)\neq \emptyset \\ 
	\{ \emptyset\} &\mbox{ if }\Gamma(3)= \emptyset.
\end{array}\right. \]

\section{HS-integral normal oriented Cayley graphs}\label{hs-ori-integral}

Let $\Irr(\Gamma)=\{ \chi_1,\ldots,\chi_h \}$. Let $E$ be the matrix $[E_{jg}]$ of size $h \times n$, whose rows are indexed by $1,\ldots,h$, and columns are indexed by the elements of $\Gamma$ such that $E_{jg}=\chi_{j}(g)$. Note that $EE^*=nI_h$ and the rank of $E$ is $h$, where $E^*$ is the conjugate transpose of $E$. 

It is well known that $\Gal(\mathbb{Q}(\omega_m)/\mathbb{Q})=\{ \sigma_r \colon r\in G_m(1), \sigma_r(\omega_m)=\omega_m^r\}$. For example, see Section 14.5 in \cite{dummitandfoote}. If $m\equiv 0 \Mod 3$, then  $\mathbb{Q}(\omega_3,\omega_m)=\mathbb{Q}(\omega_m)$. Therefore, the Galois group $\Gal(\mathbb{Q}(\omega_3,\omega_m)/\mathbb{Q}(\omega_3))$ is a subgroup of $\Gal(\mathbb{Q}(\omega_m)/\mathbb{Q})$. Thus $\Gal(\mathbb{Q}(\omega_3,\omega_m)/\mathbb{Q}(\omega_3))$ contains  those automorphisms in \linebreak[4] $\Gal(\mathbb{Q}(\omega_m)/\mathbb{Q})$ that fix $\omega_3$. Note that $G_m(1)=G_{m,3}^1(1)\cup G_{m,3}^2(1)$, a disjoint union. Using $\sigma_r(\omega_3)=\omega_3$ for all $r\in G_{m,3}^1(1)$ and $\sigma_r(\omega_3)=\omega_3^2$ for all $r \in G_{m,3}^2(1)$, we get
\[\Gal(\mathbb{Q}(\omega_3,\omega_m)/\mathbb{Q}(\omega_3))= \Gal(\mathbb{Q}(\omega_m)/\mathbb{Q}(\omega_3)) =\{ \sigma_r \colon r\in G_{m,3}^1(1), \sigma_r(\omega_m)=\omega_m^r\}.\]
If $m\not\equiv 0 \Mod 3$, then 
$[\mathbb{Q}(\omega_3,\omega_m) : \mathbb{Q}(\omega_3)]= \varphi(m).$ 
Thus the field $\mathbb{Q}(\omega_3,\omega_m)$ is a Galois extension of $\mathbb{Q}(\omega_3)$ of degree $\varphi(m)$. Any automorphism of the field $\mathbb{Q}(\omega_3,\omega_m)$ is uniquely determined by its action on $\omega_m$. Hence $$\Gal(\mathbb{Q}(\omega_3,\omega_m)/\mathbb{Q}(\omega_3))=\{ \tau_r \colon r\in G_m(1), \tau_r(\omega_m)=\omega_m^r \text{ and }\tau_r(\omega_3)=\omega_3 \}.$$  

Let $g \in \Gamma$, $m=\ord(g)$ and $\chi$ be a character of $\Gamma$. By Theorem~\ref{NewThmChap1Added}, $\chi(g)= \sum_{i=1}^{k}\epsilon_i$, where $\epsilon_1,\ldots , \epsilon_k$ are some $m$-th roots of unity. If $m \equiv 0 \Mod 3$ and $\sigma_r \in \Gal(\mathbb{Q}(\omega_3,\omega_m)/\mathbb{Q}(\omega_3))$, then 
\begin{align*}
\sigma_r (\chi(g))  =  \sigma_r \left(\sum_{i=1}^{k}\epsilon_i \right)  =  \sum_{i=1}^{k}\sigma_r (\epsilon_i) =  \sum_{i=1}^{k}\epsilon_i^r = \chi(g^r).
\end{align*}
Similarly, if $m \not\equiv 0 \Mod 3$ and $\tau_r \in \Gal(\mathbb{Q}(\omega_3,\omega_m)/\mathbb{Q}(\omega_3))$, then also $\tau_r(\chi(g)) = \chi(g^r)$.

\begin{theorem}\label{MainTheoremIntCharaCh5}
Let $\Gamma$ be a finite group and $\Irr(\Gamma)=\{ \chi_1,\ldots,\chi_h \}$. If $x=\sum\limits_{g\in \Gamma} c_g g \in \mathbb{Q}(\omega_3)\Gamma$, then $\chi_j(x)$ is rational for each $j \in \{ 1,\ldots , h \}$ if and only if the following conditions hold:
\begin{enumerate}[label=(\roman*)]
\item $\sum\limits_{s\in \Cl(g_1)}c_s=\sum\limits_{s\in \Cl(g_2)}c_s$ $\hspace{0.1cm}$ for each $g_1,g_2\in \Gamma(3)$ and $g_1\simeq g_2$;
\item $\sum\limits_{s\in \Cl(g_1)}c_s=\sum\limits_{s\in \Cl(g_2)}c_s$ $\hspace{0.1cm}$ for each $g_1,g_2\in \Gamma \setminus \Gamma(3)$ and $g_1\sim g_2$;
\item $\sum\limits_{s\in \Cl(g)}c_s=\sum\limits_{s\in \Cl(g^{-1})}\overline{c}_s$ $\hspace{0.1cm}$ for each $g \in \Gamma$.
\end{enumerate}
\end{theorem}
\begin{proof}
Let $L$ be a set of representatives of the conjugacy classes in $\Gamma$. Since characters are class functions, we have 
\begin{align}
\chi_j(x)=\sum_{g\in L} \bigg( \sum_{s\in \Cl(g)}  c_s\bigg) \chi_j(g) \textnormal{ for each } j \in \{ 1,\ldots , h \}.\label{ConjEqThetaEqua5Ch5}
\end{align} 
Assume that $\chi_j(x) \in \mathbb{Q}$ for each $j \in \{ 1,\ldots , h \}$. 
Let $g_1,g_2\in \Gamma(3)$, $g_1 \simeq g_2$ and $m=\ord(g_1)$. Therefore, there exist $r\in G_{m,3}^1(1)$ and $\sigma_r \in \Gal(\mathbb{Q}(\omega_m)/\mathbb{Q}(\omega_3))$ such that $g_2=g_1^r$ and $\sigma_r(\omega_m)=\omega_m^r$. Note that $\sigma_r(\chi_j(g_1))= \chi_j(g_1^r)$ for each $j \in \{ 1,\ldots , h \}$. For $t\in \Gamma$, let $\theta_t=\sum\limits_{j=1}^{h} \chi_j(t) \overline{\chi}_j$, where $\overline{\chi}_j(g)=\overline{\chi_j(g)}$ for each $g\in \Gamma$. By Theorem~\ref{OrthoCharaRepr}, we have
\begin{align}
	\theta_t(u)= \left\{ \begin{array}{ll}
		|C_{\Gamma}(t)| & \mbox{ if $u$ and $t$ are conjugates to each other} \\
		0 &   \mbox{ otherwise}. 
	\end{array}\right.\nonumber
\end{align} So $\theta_t(x)= |C_{\Gamma}(t)| \sum\limits_{s\in \Cl(t)} c_s \in \mathbb{Q}(\omega_3)$, and it gives that $\sigma_r(\theta_t(x))=\theta_t(x)$. Since $\chi_j(x)$ is assumed to be a rational number, we have $\sigma_r(\chi_j(x))=\chi_j(x)$ for each $j\in \{1,\ldots,h\}$. Thus 
\begin{align}
|C_{\Gamma}(g_1)| \sum\limits_{s\in \Cl(g_1)} c_s=\theta_{g_1}(x)= \sigma_r(\theta_{g_1}(x)) &=\sum\limits_{j=1}^{h} \sigma_r(\chi_j(g_1)) \sigma_r(\overline{\chi}_j(x)) \nonumber\\
&=\sum\limits_{j=1}^{h} \chi_j(g_1^r) \overline{\chi}_j(x) \nonumber\\
&=\theta_{g_1^r}(x) =\theta_{g_2}(x)=|C_{\Gamma}(g_2)| \sum\limits_{s\in \Cl(g_2)}c_s. \label{ConjEqThetaEquaCh5}
\end{align}
Since $g_1\simeq g_2$, we have $C_{\Gamma}(g_1)=C_{\Gamma}(g_2)$. So Equation $(\ref{ConjEqThetaEquaCh5})$ implies that $\sum\limits_{s\in \Cl(g_1)}c_s=\sum\limits_{s\in \Cl(g_2)}c_s$. Hence condition (i) holds.

Now let $g_1,g_2\in \Gamma \setminus \Gamma(3)$, $g_1 \sim g_2$, and $m=\ord(g_1)$. Then there is $r\in G_{m}(1)$ and \linebreak[4] $\tau_r \in \Gal(\mathbb{Q}(\omega_3,\omega_m)/\mathbb{Q}(\omega_3))$ such that $g_2=g_1^r$, $\tau_r(\omega_m)=\omega_m^r$ and $\tau_r(\omega_3)=\omega_3$. Now proceeding as in the proof of condition (i), we have $\sum\limits_{s\in \Cl(g_1)}c_s=\sum\limits_{s\in \Cl(g_2)}c_s$. Thus condition (ii) also holds.

Again
\begin{align}
0=\chi_j(x)- \overline{\chi_j(x)}&=\sum_{g\in L} \bigg( \sum_{s\in \Cl(g)}  c_s\bigg) \chi_j(g) - \sum_{g\in L} \bigg( \sum_{s\in \Cl(g)} \overline{c}_s\bigg) \overline{\chi_j(g)}\nonumber \\
&=\sum_{g\in L} \bigg( \sum_{s\in \Cl(g)}  c_s\bigg) \chi_j(g) - \sum_{g\in L} \bigg( \sum_{s\in \Cl(g)} \overline{c}_s\bigg) \chi_j(g^{-1})\nonumber\\
&=\sum_{g\in L}  \bigg( \sum_{s\in \Cl(g)}  c_s - \sum_{s\in \Cl(g^{-1})}  \overline{c}_s  \bigg) \chi_j(g),\nonumber
\end{align} and so 
\begin{align}
\sum_{g\in L}  \bigg( \sum_{s\in \Cl(g)}  c_s - \sum_{s\in \Cl(g^{-1})}  \overline{c}_s  \bigg) \begin{bmatrix}
			\chi_1(g) \\
			\vdots \\
			\chi_h(g)
		\end{bmatrix}=\begin{bmatrix}
			0 \\
			\vdots \\
			0
		\end{bmatrix}.\label{eqLIsumZeroCh5}
\end{align} 
Note that the number of irreducible characters of $\Gamma$ is equal to the number of conjugacy classes of $\Gamma$, that is, $|L|=h$.
Since characters are class functions and rank of $E$ is $h$, the columns of $E$ corresponding to the elements of $L$ are linearly independent. Thus by Equation $(\ref{eqLIsumZeroCh5})$, 
$ \sum\limits_{s\in \Cl(g)}  c_s - \sum\limits_{s\in \Cl(g^{-1})}  \overline{c}_s =0$ for all $g\in L$, and so condition (iii) holds.


Conversely, assume that the three conditions of the theorem hold. Let $n$ be the number of elements of $\Gamma$. We have the following two cases.

\noindent\textbf{Case 1.} Assume that $n\equiv 0 \Mod 3$.
Let $\sigma_k \in \Gal(\mathbb{Q}(\omega_3,\omega_n)/\mathbb{Q}(\omega_3))$. Then $\sigma_k(\omega_n)=\omega_n^k$ and $k\in G_{n,3}^1(1)$, and so $\sigma_k(\chi_j(g))=\chi_j(g^k)$ for each $j \in \{ 1,\ldots , h \}$. Thus
\begin{align}
\sigma_k(\chi_j(x)) &= \sum_{g\in L} \bigg( \sum_{s\in \Cl(g)} c_s\bigg) \sigma_k(\chi_j(g))\nonumber\\
&= \sum_{g\in L} \bigg( \sum_{s\in \Cl(g)} c_s\bigg) \chi_j(g^k).
\label{ConjEqThetaEqua3Ch5}
\end{align} 
In the sum of Equation~(\ref{ConjEqThetaEqua3Ch5}) we have two possible casses, namely, $g \in \Gamma(3)$  or $ g \in \Gamma \setminus \Gamma(3)$. If $g \in \Gamma(3)$, then using the fact $g\simeq g^k$ and condition (i), we get  $\sum\limits_{s\in \Cl(g)}c_s=\sum\limits_{s\in \Cl(g^k)}c_s$. Similarly, if $ g \in \Gamma \setminus \Gamma(3)$, then using the fact $g\sim g^k$ and condition (ii), we get  $\sum\limits_{s\in \Cl(g)}c_s=\sum\limits_{s\in \Cl(g^k)}c_s$. Therefore, we have $\sum\limits_{s\in \Cl(g)}c_s=\sum\limits_{s\in \Cl(g^k)}c_s$ for each $g \in \Gamma$. Now from Equation $(\ref{ConjEqThetaEqua3Ch5})$, we get 
\begin{align}
\sigma_k(\chi_j(x)) &= \sum_{g\in L} \bigg( \sum_{s\in \Cl(g^k)} c_s\bigg) \chi_j(g^k)=\chi_j(x). \label{neweqCh5}
\end{align} 
The second equality in Equation $(\ref{neweqCh5})$ holds, because $\{ g^k\colon g \in L\}$ is also a set of representatives of conjugacy classes of $\Gamma$. Now since $\sigma_k(\chi_j(x)) = \chi_j(x)$ for each $k \in G_{n,3}^1(1)$, we have that $\chi_j(x)\in \mathbb{Q}(\omega_3)$. 

\noindent\textbf{Case 2.} Assume that $n\not\equiv 0 \Mod 3$. Let $\tau_r \in \Gal(\mathbb{Q}(\omega_3,\omega_n)/\mathbb{Q}(\omega_3))$. Then we have $\tau_r(\chi_j(g))=\chi_j(g^r)$ for each $j \in \{ 1,\ldots , h \}$. Note that $g \sim g^r$. Therefore using Equation (\ref{ConjEqThetaEqua5Ch5}) and condition (ii), we have
\begin{align*}
\tau_r(\chi_j(x)) &= \sum_{g\in L} \bigg( \sum_{s\in \Cl(g)} c_s\bigg) \tau_r(\chi_j(g))\nonumber\\
&= \sum_{g\in L} \bigg( \sum_{s\in \Cl(g)} c_s\bigg) \chi_j(g^r)\\
&= \sum_{g\in L} \bigg( \sum_{s\in \Cl(g^r)} c_s\bigg) \chi_j(g^r)\\
&= \chi_j(x).
\end{align*}
This gives that $\chi_j(x)\in \mathbb{Q}(\omega_3)$. Thus in both the cases, we get $\chi_j(x)\in \mathbb{Q}(\omega_3)$. Taking complex conjugates in Equation $(\ref{ConjEqThetaEqua5Ch5})$, we get
\begin{align}
\overline{\chi_j(x)}=\sum_{g\in L} \bigg( \sum_{s\in \Cl(g)} \overline{c}_s\bigg) \overline{\chi_j(g)} &= \sum_{g\in L} \bigg( \sum_{s\in \Cl(g)} \overline{c}_s\bigg) \chi_j(g^{-1}) \nonumber\\
&= \sum_{g\in L} \bigg( \sum_{s\in \Cl(g^{-1})} c_s\bigg) \chi_j(g^{-1})\nonumber\\
&= \chi_j(x). \label{ConjEqThetaEqua4Ch5}
\end{align}
Equation $(\ref{ConjEqThetaEqua4Ch5})$ implies that $\chi_j(x)\in \mathbb{Q}$ for all $j \in \{ 1, \ldots , h\}$.
\end{proof}

Indeed, we can replace condition (i) of Theorem \ref{MainTheoremIntCharaCh5} by $\sum\limits_{s\in \Cl(x)}c_s=\sum\limits_{s\in \Cl(y)}c_s$ for all $x,y \in \langle\!\langle g \rangle\!\rangle$ and $g \in \Gamma(3)$.

\begin{theorem}\label{NorOriCayGraphIntegCh5}
Let $\Gamma$ be a finite group and $\Cay(\Gamma, S)$ be a normal oriented Cayley graph. Then $\Cay(\Gamma, S)$ is HS-integral if and only if $S \in \mathbb{E}(\Gamma)$.
\end{theorem}
\begin{proof}
Let $\Irr(\Gamma)=\{ \chi_1,\ldots,\chi_h \}$ and $x=\sum\limits_{g\in \Gamma}  c_g g$, where 
$$c_g= \left\{ \begin{array}{rl}
		-\omega_3^2 & \mbox{if } g\in S  \\
		-\omega_3 & \mbox{if } g\in S^{-1}\\
		0 &   \mbox{otherwise}. 
	\end{array}\right.$$ 
Note that $-\omega_3^2=\omega_6$ and $-\omega_3=\omega_6^5$. Thus $\chi_j(x)= \sum\limits_{s\in S}  (-\omega_3^2\chi_j(s)-\omega_3 \chi_j(s^{-1}))$, and so $\frac{\chi_j(x)}{\chi_j(\mathbf 1)}$ is an HS-eigenvalue of $\Cay(\Gamma, S)$. Assume that the normal oriented Cayley graph $\Cay(\Gamma, S)$ is HS-integral. Thus $\chi_j(x)$ is an integer for each $j \in \{ 1,\ldots , h \}$, and therefore the three conditions of Theorem \ref{MainTheoremIntCharaCh5} are satisfied for $x$. Using the fact that $g \sim g^{-1}$, and conditions (ii) and (iii) of Theorem \ref{MainTheoremIntCharaCh5}, we get $\Im \left(\sum\limits_{s\in \Cl(g)}c_s \right)=0$ for all $g \in \Gamma\setminus \Gamma (3)$. Note that $S$ is a union of some conjugacy classes of $\Gamma$. Therefore, if $g\in S$ then $\Cl(g)\subseteq S$, and so by the definition of $c_g$, we get $\Im \left(\sum\limits_{s\in \Cl(g)}c_s\right)=\frac{\sqrt{3}|\Cl(g)|}{2}\neq 0$. Thus $S \cap (\Gamma \setminus \Gamma(3))=\emptyset$, that is, $S\subseteq \Gamma(3)$. Again, let  $g_1 \in S$, $g_2 \in \Gamma(3)$ and $ g_1 \simeq g_2$. By the first condition of Theorem \ref{MainTheoremIntCharaCh5}, we get $0 \neq \sum\limits_{s\in \Cl(g_1)}c_s=\sum\limits_{s\in \Cl(g_2)}c_s$, which implies that $g_2\in S$. Thus $g_1 \in S$ gives $\langle\!\langle g_1 \rangle\!\rangle \subseteq S$. Hence $S\in \mathbb{E}(\Gamma)$. 
	
Conversely, assume that $S\in \mathbb{E}(\Gamma)$. Let $\Cay(\Gamma, S)$ be a normal oriented Cayley graph, so that $S$ is a union of some conjugacy classes of $\Gamma$.
Let \[S= \langle\!\langle x_1 \rangle\!\rangle \cup \cdots \cup \langle\!\langle x_r \rangle\!\rangle = \Cl(y_1) \cup \cdots \cup \Cl(y_k) \subseteq \Gamma(3)\]
 for some $x_1,\ldots,x_r, y_1,\ldots,y_k\in \Gamma(3)$. We have $$S^{-1}= \langle\!\langle x_1^{-1} \rangle\!\rangle \cup\cdots \cup \langle\!\langle x_r^{-1} \rangle\!\rangle = \Cl(y_1^{-1}) \cup \cdots \cup \Cl(y_k^{-1}) \subseteq \Gamma(3). $$

Now for $g_1,g_2\in \Gamma(3)$, if $g_1\simeq g_2$ then $\Cl(g_1),\Cl(g_2) \subseteq S$ or $\Cl(g_1),\Cl(g_2) \subseteq S^{-1}$ or  $\Cl(g_1),\Cl(g_2) \subseteq (S \cup S^{-1})^{c}$. Note that $|\Cl(g_1)|=|\Cl(g_2)|$. For all the cases, using the definition of $c_g$, we find 
$$\sum\limits_{s\in \Cl(g_1)}c_s=\sum\limits_{s\in \Cl(g_2)}c_s.$$ 
Thus condition (i) of Theorem~\ref{MainTheoremIntCharaCh5} holds. If  $g_1,g_2\in \Gamma \setminus \Gamma(3)$ and $g_1 \sim g_2$, then clearly $\Cl(g_1),\Cl(g_2) \subseteq \Gamma \setminus \Gamma(3)$. Therefore $\Cl(g_1),\Cl(g_2) \subseteq (S \cup S^{-1})^c$. Accordingly, $$\sum\limits_{s\in \Cl(g_1)}c_s = 0 = \sum\limits_{s\in \Cl(g_2)}c_s.$$ Hence condition (ii) of Theorem~\ref{MainTheoremIntCharaCh5} also holds.

Again for $g \in \Gamma$, we have $\Cl(g) \subseteq S$ if and only if $\Cl(g^{-1}) \subseteq S^{-1}$. Therefore we have $\sum\limits_{s\in \Cl(g)}c_s=\sum\limits_{s\in \Cl(g^{-1})}\overline{c}_s$, and so condition (iii) of Theorem~\ref{MainTheoremIntCharaCh5} also holds. Thus by Theorem~\ref{MainTheoremIntCharaCh5}, $\chi_j(x)$ is a rational number for each $j \in \{ 1,\ldots , h \}$. Consequently, the HS-eigenvalue $\mu_j := \frac{\chi_j(x)}{\chi_j(\mathbf 1)}$ of $\Cay(\Gamma, S)$ is a rational algebraic integer, and hence an integer for each $j \in \{ 1,\ldots , h \}$.
\end{proof}


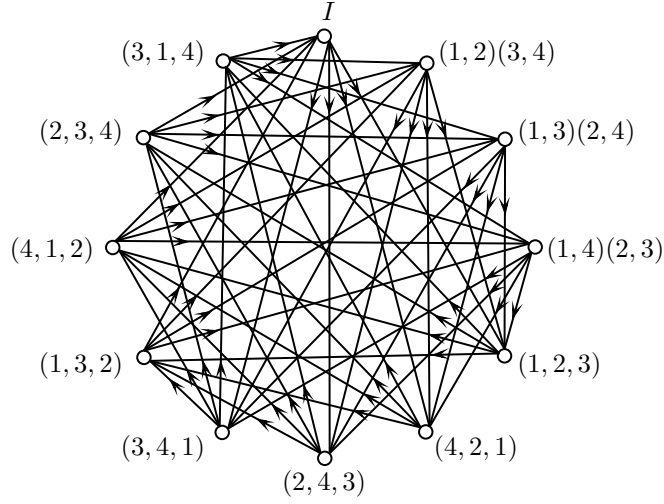
\begin{figure}[ht]
\centering

\tikzset{every picture/.style={line width=0.75pt}} 

\begin{tikzpicture}[x=0.2pt,y=0.2pt,yscale=-1,xscale=1]

\draw    (841.32,693.15) -- (509.58,109.42) ;
\draw    (699.58,840.42) -- (509.58,109.42) ;
\draw    (510.58,890.42) -- (509.58,109.42) ;
\draw    (318.58,843.42) -- (509.58,109.42) ;
\draw    (159.32,696.15) -- (487.6,100.86) ;
\draw    (105.58,488.42) -- (487.6,100.86) ;
\draw    (167.58,285.42) -- (487.6,100.86) ;
\draw    (319.58,142.42) -- (487.6,100.86) ;
\draw    (841.32,693.15) -- (694.32,164.58) ;
\draw    (699.58,840.42) -- (694.32,164.58) ;
\draw    (510.58,890.42) -- (694.32,164.58) ;
\draw    (318.58,843.42) -- (694.32,164.58) ;
\draw    (319.58,142.42) -- (681.6,151.86) ;
\draw    (167.58,285.42) -- (681.6,151.86) ;
\draw    (105.58,488.42) -- (681.6,151.86) ;
\draw    (159.32,696.15) -- (681.6,151.86) ;
\draw  [color={rgb, 255:red, 0; green, 0; blue, 0 }  ,draw opacity=1 ][line width=0.75]  (487.6,100.86) .. controls (487.6,93.84) and (493.3,88.15) .. (500.32,88.15) .. controls (507.35,88.15) and (513.04,93.84) .. (513.04,100.86) .. controls (513.04,107.89) and (507.35,113.58) .. (500.32,113.58) .. controls (493.3,113.58) and (487.6,107.89) .. (487.6,100.86) -- cycle ;
\draw  [color={rgb, 255:red, 0; green, 0; blue, 0 }  ,draw opacity=1 ][line width=0.75]  (681.6,151.86) .. controls (681.6,144.84) and (687.3,139.15) .. (694.32,139.15) .. controls (701.35,139.15) and (707.04,144.84) .. (707.04,151.86) .. controls (707.04,158.89) and (701.35,164.58) .. (694.32,164.58) .. controls (687.3,164.58) and (681.6,158.89) .. (681.6,151.86) -- cycle ;
\draw  [color={rgb, 255:red, 0; green, 0; blue, 0 }  ,draw opacity=1 ][line width=0.75]  (829.6,295.86) .. controls (829.6,288.84) and (835.3,283.15) .. (842.32,283.15) .. controls (849.35,283.15) and (855.04,288.84) .. (855.04,295.86) .. controls (855.04,302.89) and (849.35,308.58) .. (842.32,308.58) .. controls (835.3,308.58) and (829.6,302.89) .. (829.6,295.86) -- cycle ;
\draw  [color={rgb, 255:red, 0; green, 0; blue, 0 }  ,draw opacity=1 ][line width=0.75]  (296.6,147.86) .. controls (296.6,140.84) and (302.3,135.15) .. (309.32,135.15) .. controls (316.35,135.15) and (322.04,140.84) .. (322.04,147.86) .. controls (322.04,154.89) and (316.35,160.58) .. (309.32,160.58) .. controls (302.3,160.58) and (296.6,154.89) .. (296.6,147.86) -- cycle ;
\draw  [color={rgb, 255:red, 0; green, 0; blue, 0 }  ,draw opacity=1 ][line width=0.75]  (145.6,292.86) .. controls (145.6,285.84) and (151.3,280.15) .. (158.32,280.15) .. controls (165.35,280.15) and (171.04,285.84) .. (171.04,292.86) .. controls (171.04,299.89) and (165.35,305.58) .. (158.32,305.58) .. controls (151.3,305.58) and (145.6,299.89) .. (145.6,292.86) -- cycle ;
\draw  [color={rgb, 255:red, 0; green, 0; blue, 0 }  ,draw opacity=1 ][line width=0.75]  (87.6,500.86) .. controls (87.6,493.84) and (93.3,488.15) .. (100.32,488.15) .. controls (107.35,488.15) and (113.04,493.84) .. (113.04,500.86) .. controls (113.04,507.89) and (107.35,513.58) .. (100.32,513.58) .. controls (93.3,513.58) and (87.6,507.89) .. (87.6,500.86) -- cycle ;
\draw  [color={rgb, 255:red, 0; green, 0; blue, 0 }  ,draw opacity=1 ][line width=0.75]  (146.6,708.86) .. controls (146.6,701.84) and (152.3,696.15) .. (159.32,696.15) .. controls (166.35,696.15) and (172.04,701.84) .. (172.04,708.86) .. controls (172.04,715.89) and (166.35,721.58) .. (159.32,721.58) .. controls (152.3,721.58) and (146.6,715.89) .. (146.6,708.86) -- cycle ;
\draw  [color={rgb, 255:red, 0; green, 0; blue, 0 }  ,draw opacity=1 ][line width=0.75]  (294.6,850.86) .. controls (294.6,843.84) and (300.3,838.15) .. (307.32,838.15) .. controls (314.35,838.15) and (320.04,843.84) .. (320.04,850.86) .. controls (320.04,857.89) and (314.35,863.58) .. (307.32,863.58) .. controls (300.3,863.58) and (294.6,857.89) .. (294.6,850.86) -- cycle ;
\draw  [color={rgb, 255:red, 0; green, 0; blue, 0 }  ,draw opacity=1 ][line width=0.75]  (488.6,899.86) .. controls (488.6,892.84) and (494.3,887.15) .. (501.32,887.15) .. controls (508.35,887.15) and (514.04,892.84) .. (514.04,899.86) .. controls (514.04,906.89) and (508.35,912.58) .. (501.32,912.58) .. controls (494.3,912.58) and (488.6,906.89) .. (488.6,899.86) -- cycle ;
\draw  [color={rgb, 255:red, 0; green, 0; blue, 0 }  ,draw opacity=1 ][line width=0.75]  (679.6,849.86) .. controls (679.6,842.84) and (685.3,837.15) .. (692.32,837.15) .. controls (699.35,837.15) and (705.04,842.84) .. (705.04,849.86) .. controls (705.04,856.89) and (699.35,862.58) .. (692.32,862.58) .. controls (685.3,862.58) and (679.6,856.89) .. (679.6,849.86) -- cycle ;
\draw  [color={rgb, 255:red, 0; green, 0; blue, 0 }  ,draw opacity=1 ][line width=0.75]  (828.6,705.86) .. controls (828.6,698.84) and (834.3,693.15) .. (841.32,693.15) .. controls (848.35,693.15) and (854.04,698.84) .. (854.04,705.86) .. controls (854.04,712.89) and (848.35,718.58) .. (841.32,718.58) .. controls (834.3,718.58) and (828.6,712.89) .. (828.6,705.86) -- cycle ;
\draw  [color={rgb, 255:red, 0; green, 0; blue, 0 }  ,draw opacity=1 ][line width=0.75]  (886.6,499.86) .. controls (886.6,492.84) and (892.3,487.15) .. (899.32,487.15) .. controls (906.35,487.15) and (912.04,492.84) .. (912.04,499.86) .. controls (912.04,506.89) and (906.35,512.58) .. (899.32,512.58) .. controls (892.3,512.58) and (886.6,506.89) .. (886.6,499.86) -- cycle ;
\draw  [fill={rgb, 255:red, 0; green, 0; blue, 0 }  ,fill opacity=1 ] (485.48,222.16) -- (476.84,235.86) -- (476.57,219.67) -- (478.93,228.39) -- cycle ;
\draw    (841.32,693.15) -- (842.32,308.58) ;
\draw    (699.58,840.42) -- (842.32,308.58) ;
\draw    (510.58,890.42) -- (842.32,308.58) ;
\draw    (318.58,843.42) -- (842.32,308.58) ;
\draw    (841.32,693.15) -- (892.58,509.42) ;
\draw    (699.58,840.42) -- (892.58,509.42) ;
\draw    (510.58,890.42) -- (892.58,509.42) ;
\draw    (318.58,843.42) -- (892.58,509.42) ;
\draw    (319.58,142.42) -- (829.6,295.86) ;
\draw    (167.58,285.42) -- (829.6,295.86) ;
\draw    (105.58,488.42) -- (829.6,295.86) ;
\draw    (159.32,696.15) -- (829.6,295.86) ;
\draw    (319.58,142.42) -- (888.58,493.42) ;
\draw    (167.58,285.42) -- (888.58,493.42) ;
\draw    (105.58,488.42) -- (888.58,493.42) ;
\draw    (159.32,696.15) -- (888.58,493.42) ;
\draw    (314.58,158.42) -- (831.58,701.42) ;
\draw    (163.58,303.42) -- (831.58,701.42) ;
\draw    (112.58,506.42) -- (831.58,701.42) ;
\draw    (170.58,715.42) -- (831.58,701.42) ;
\draw    (314.58,158.42) -- (682.58,841.42) ;
\draw    (163.58,303.42) -- (682.58,841.42) ;
\draw    (112.58,506.42) -- (682.58,841.42) ;
\draw    (170.58,715.42) -- (682.58,841.42) ;
\draw    (314.58,158.42) -- (493.58,890.42) ;
\draw    (163.58,303.42) -- (493.58,890.42) ;
\draw    (112.58,506.42) -- (493.58,890.42) ;
\draw    (170.58,715.42) -- (493.58,890.42) ;
\draw    (314.58,158.42) -- (302.58,840.42) ;
\draw    (163.58,303.42) -- (302.58,840.42) ;
\draw    (112.58,506.42) -- (302.58,840.42) ;
\draw    (170.58,715.42) -- (302.58,840.42) ;
\draw  [fill={rgb, 255:red, 0; green, 0; blue, 0 }  ,fill opacity=1 ] (514.68,227.7) -- (509.81,243.15) -- (505.43,227.56) -- (509.93,235.39) -- cycle ;
\draw  [fill={rgb, 255:red, 0; green, 0; blue, 0 }  ,fill opacity=1 ] (540.93,213.58) -- (541.34,229.77) -- (532.13,216.45) -- (538.93,222.39) -- cycle ;
\draw  [fill={rgb, 255:red, 0; green, 0; blue, 0 }  ,fill opacity=1 ] (562.15,193.36) -- (565.75,209.15) -- (554.09,197.91) -- (561.93,202.39) -- cycle ;
\draw  [fill={rgb, 255:red, 0; green, 0; blue, 0 }  ,fill opacity=1 ] (644.72,264.8) -- (633.21,276.2) -- (636.6,260.36) -- (636.93,269.39) -- cycle ;
\draw  [fill={rgb, 255:red, 0; green, 0; blue, 0 }  ,fill opacity=1 ] (672.32,270.99) -- (664.04,284.91) -- (663.34,268.73) -- (665.93,277.39) -- cycle ;
\draw  [fill={rgb, 255:red, 0; green, 0; blue, 0 }  ,fill opacity=1 ] (700.04,272.93) -- (694.45,288.13) -- (690.8,272.35) -- (694.93,280.39) -- cycle ;
\draw  [fill={rgb, 255:red, 0; green, 0; blue, 0 }  ,fill opacity=1 ] (728.55,271.5) -- (729.65,287.66) -- (719.88,274.74) -- (726.93,280.39) -- cycle ;
\draw  [fill={rgb, 255:red, 0; green, 0; blue, 0 }  ,fill opacity=1 ] (781.43,376.31) -- (768.02,385.39) -- (774.27,370.45) -- (772.93,379.39) -- cycle ;
\draw  [fill={rgb, 255:red, 0; green, 0; blue, 0 }  ,fill opacity=1 ] (796.1,399.53) -- (783.61,409.83) -- (788.42,394.37) -- (787.93,403.39) -- cycle ;
\draw  [fill={rgb, 255:red, 0; green, 0; blue, 0 }  ,fill opacity=1 ] (818.69,415.39) -- (809.59,428.79) -- (809.86,412.59) -- (811.93,421.39) -- cycle ;
\draw  [fill={rgb, 255:red, 0; green, 0; blue, 0 }  ,fill opacity=1 ] (846.08,418.36) -- (842.41,434.13) -- (836.84,418.93) -- (841.93,426.39) -- cycle ;
\draw  [fill={rgb, 255:red, 0; green, 0; blue, 0 }  ,fill opacity=1 ] (823.97,554.5) -- (808.22,558.28) -- (819.33,546.49) -- (814.93,554.39) -- cycle ;
\draw  [fill={rgb, 255:red, 0; green, 0; blue, 0 }  ,fill opacity=1 ] (829.85,578.92) -- (815,585.39) -- (823.88,571.85) -- (820.93,580.39) -- cycle ;
\draw  [fill={rgb, 255:red, 0; green, 0; blue, 0 }  ,fill opacity=1 ] (839.01,610.63) -- (826.76,621.21) -- (831.22,605.65) -- (830.93,614.68) -- cycle ;
\draw  [fill={rgb, 255:red, 0; green, 0; blue, 0 }  ,fill opacity=1 ] (866.59,618.57) -- (857.71,632.11) -- (857.73,615.91) -- (859.93,624.68) -- cycle ;
\draw  [fill={rgb, 255:red, 0; green, 0; blue, 0 }  ,fill opacity=1 ] (724.04,707.66) -- (708.2,704.3) -- (723.29,698.43) -- (715.93,703.68) -- cycle ;
\draw  [fill={rgb, 255:red, 0; green, 0; blue, 0 }  ,fill opacity=1 ] (716.51,674.87) -- (702.36,666.99) -- (718.51,665.84) -- (709.93,668.68) -- cycle ;
\draw  [fill={rgb, 255:red, 0; green, 0; blue, 0 }  ,fill opacity=1 ] (723.57,642.43) -- (712.11,630.99) -- (727.96,634.29) -- (718.93,634.68) -- cycle ;
\draw  [fill={rgb, 255:red, 0; green, 0; blue, 0 }  ,fill opacity=1 ] (746.93,619.49) -- (739.58,605.06) -- (753.64,613.1) -- (744.93,610.68) -- cycle ;
\draw  [fill={rgb, 255:red, 0; green, 0; blue, 0 }  ,fill opacity=1 ] (621.56,737.7) -- (618.24,721.85) -- (629.7,733.3) -- (621.93,728.68) -- cycle ;
\draw  [fill={rgb, 255:red, 0; green, 0; blue, 0 }  ,fill opacity=1 ] (602.54,765.57) -- (595.84,750.82) -- (609.52,759.49) -- (600.93,756.68) -- cycle ;
\draw  [fill={rgb, 255:red, 0; green, 0; blue, 0 }  ,fill opacity=1 ] (589.07,791.71) -- (578.35,779.57) -- (593.97,783.86) -- (584.93,783.68) -- cycle ;
\draw  [fill={rgb, 255:red, 0; green, 0; blue, 0 }  ,fill opacity=1 ] (571.78,819.56) -- (558.59,810.17) -- (574.77,810.81) -- (565.93,812.68) -- cycle ;
\draw  [fill={rgb, 255:red, 0; green, 0; blue, 0 }  ,fill opacity=1 ] (462.83,782.46) -- (462.62,766.27) -- (471.67,779.7) -- (464.93,773.68) -- cycle ;
\draw  [fill={rgb, 255:red, 0; green, 0; blue, 0 }  ,fill opacity=1 ] (438.56,801.7) -- (435.24,785.85) -- (446.7,797.3) -- (438.93,792.68) -- cycle ;
\draw  [fill={rgb, 255:red, 0; green, 0; blue, 0 }  ,fill opacity=1 ] (411.98,815.48) -- (404.55,801.09) -- (418.65,809.06) -- (409.93,806.68) -- cycle ;
\draw  [fill={rgb, 255:red, 0; green, 0; blue, 0 }  ,fill opacity=1 ] (387.43,838.51) -- (376.17,826.87) -- (391.97,830.45) -- (382.93,830.68) -- cycle ;
\draw  [fill={rgb, 255:red, 0; green, 0; blue, 0 }  ,fill opacity=1 ] (300.5,735.55) -- (304.75,719.92) -- (309.75,735.32) -- (304.93,727.68) -- cycle ;
\draw  [fill={rgb, 255:red, 0; green, 0; blue, 0 }  ,fill opacity=1 ] (271.3,736.32) -- (272.07,720.14) -- (280.29,734.1) -- (273.93,727.68) -- cycle ;
\draw  [fill={rgb, 255:red, 0; green, 0; blue, 0 }  ,fill opacity=1 ] (247.31,751.7) -- (242.68,736.18) -- (255.06,746.63) -- (246.93,742.68) -- cycle ;
\draw  [fill={rgb, 255:red, 0; green, 0; blue, 0 }  ,fill opacity=1 ] (217.77,767.52) -- (210.69,752.96) -- (224.59,761.27) -- (215.93,758.68) -- cycle ;
\draw  [fill={rgb, 255:red, 0; green, 0; blue, 0 }  ,fill opacity=1 ] (211.79,590.58) -- (224.22,580.21) -- (219.5,595.7) -- (219.93,586.68) -- cycle ;
\draw  [fill={rgb, 255:red, 0; green, 0; blue, 0 }  ,fill opacity=1 ] (243.26,602.19) -- (257.23,594) -- (250.02,608.51) -- (251.93,599.68) -- cycle ;
\draw  [fill={rgb, 255:red, 0; green, 0; blue, 0 }  ,fill opacity=1 ] (267.91,626.21) -- (283.79,623.05) -- (272.24,634.39) -- (276.93,626.68) -- cycle ;
\draw  [fill={rgb, 255:red, 0; green, 0; blue, 0 }  ,fill opacity=1 ] (279.38,657.77) -- (295.52,659.05) -- (281.31,666.82) -- (287.93,660.68) -- cycle ;
\draw  [fill={rgb, 255:red, 0; green, 0; blue, 0 }  ,fill opacity=1 ] (192.23,394.08) -- (206.3,386.07) -- (198.91,400.48) -- (200.93,391.68) -- cycle ;
\draw  [fill={rgb, 255:red, 0; green, 0; blue, 0 }  ,fill opacity=1 ] (205.9,424.76) -- (221.56,420.64) -- (210.71,432.66) -- (214.93,424.68) -- cycle ;
\draw  [fill={rgb, 255:red, 0; green, 0; blue, 0 }  ,fill opacity=1 ] (220.19,454.39) -- (236.39,454.52) -- (222.77,463.28) -- (228.93,456.68) -- cycle ;
\draw  [fill={rgb, 255:red, 0; green, 0; blue, 0 }  ,fill opacity=1 ] (224.26,484.9) -- (239.69,489.82) -- (224.09,494.16) -- (231.93,489.68) -- cycle ;
\draw  [fill={rgb, 255:red, 0; green, 0; blue, 0 }  ,fill opacity=1 ] (274.94,217.81) -- (290.95,215.36) -- (278.9,226.18) -- (283.93,218.68) -- cycle ;
\draw  [fill={rgb, 255:red, 0; green, 0; blue, 0 }  ,fill opacity=1 ] (281.4,249.72) -- (297.53,251.1) -- (283.28,258.78) -- (289.93,252.68) -- cycle ;
\draw  [fill={rgb, 255:red, 0; green, 0; blue, 0 }  ,fill opacity=1 ] (280.27,282.9) -- (295.69,287.82) -- (280.09,292.15) -- (287.93,287.68) -- cycle ;
\draw  [fill={rgb, 255:red, 0; green, 0; blue, 0 }  ,fill opacity=1 ] (284.31,313.61) -- (297.2,323.41) -- (281.04,322.27) -- (289.93,320.68) -- cycle ;
\draw  [fill={rgb, 255:red, 0; green, 0; blue, 0 }  ,fill opacity=1 ] (415.35,113.85) -- (431.51,114.99) -- (417.37,122.88) -- (423.93,116.68) -- cycle ;
\draw  [fill={rgb, 255:red, 0; green, 0; blue, 0 }  ,fill opacity=1 ] (422.51,140.52) -- (437.68,146.21) -- (421.87,149.75) -- (429.93,145.68) -- cycle ;
\draw  [fill={rgb, 255:red, 0; green, 0; blue, 0 }  ,fill opacity=1 ] (398.39,160.54) -- (411.16,170.5) -- (395.02,169.16) -- (403.93,167.68) -- cycle ;
\draw  [fill={rgb, 255:red, 0; green, 0; blue, 0 }  ,fill opacity=1 ] (381.09,174.5) -- (391.37,187.02) -- (375.91,182.17) -- (384.93,182.68) -- cycle ;

\draw (490,30) node [anchor=north west][inner sep=0.75pt]    {$I$};
\draw (710,102) node [anchor=north west][inner sep=0.75pt]  {$( 1,2)( 3,4)$};
\draw (860,252) node [anchor=north west][inner sep=0.75pt]   {$( 1,3)( 2,4)$};
\draw (915,475) node [anchor=north west][inner sep=0.75pt]   {$( 1,4)( 2,3)$};
\draw (860,694) node [anchor=north west][inner sep=0.75pt]    {$( 1,2,3)$};
\draw (702,850) node [anchor=north west][inner sep=0.75pt]   {$( 4,2,1)$};
\draw (415,917.57) node [anchor=north west][inner sep=0.75pt]    {$( 2,4,3)$};
\draw (110,850) node [anchor=north west][inner sep=0.75pt]   {$( 3,4,1)$};
\draw (-45,694) node [anchor=north west][inner sep=0.75pt]  {$( 1,3,2)$};
\draw (-100,475) node [anchor=north west][inner sep=0.75pt]    {$( 4,1,2)$};
\draw (-45,252) node [anchor=north west][inner sep=0.75pt]  {$( 2,3,4)$};
\draw (110,102) node [anchor=north west][inner sep=0.75pt]  {$( 3,1,4)$};
\end{tikzpicture}

\caption{ The oriented graph $\Cay(A_4,\{ (1,2,3), (4,2,1), (2,4,3), (3,4,1)\})$}\label{Fig1}
\end{figure}

In the following example, we illustrate an use of Theorem~\ref{NorOriCayGraphIntegCh5}.

\begin{ex} \normalfont Consider $S=\{ (1,2,3), (4,2,1), (2,4,3), (3,4,1) \}$ in the alternating group $A_4$. The conjugacy classes of $A_4$ are $\{I\}, \Cl((1,2)(3,4)) ,  \Cl((1,2,3))$ and $\Cl((1,3,2)) $, where
\begin{align*}
&~ I = (1)(2)(3)(4),\\
&\Cl((1,2)(3,4))= \{(1,2)(3,4),(1,3)(2,4), (1,4)(2,3)\},\\
&\Cl((1,2,3))=\{ (1,2,3), (4,2,1), (2,4,3), (3,4,1)\} \text{ and}\\
&\Cl((1,3,2))=\{(1,3,2),(4,1,2),(2,3,4), (3,1,4)\}.
\end{align*}
The normal oriented Cayley graph $\Cay(A_4, S)$ is shown in Figure~\ref{Fig1}. We see that $S= \langle\!\langle (1,2,3) \rangle\!\rangle \cup \langle\!\langle (4,2,1)\rangle\!\rangle \cup \langle\!\langle (2,4,3) \rangle\!\rangle \cup \langle\!\langle (3,4,1)\rangle\!\rangle= \Cl((1,2,3))$. Therefore $S \in \mathbb{E}(\Gamma)$, and hence $\Cay(A_4, S)$ is HS-integral by Theorem~\ref{NorOriCayGraphIntegCh5}. The character table of the group $A_4$ is given in Table~\ref{table1Ch5}~\cite{james2001representations}, where $\Irr(A_4)=\{ \chi_1, \chi_2,\chi_3, \chi_4 \}$. Further, using Corollary~\ref{coro2EigNorCayMix}, the  HS-spectrum of  $\Cay(A_4, S)$ is obtained as $\{ [\mu_{1}]^{1},[\mu_{2}]^{1},[\mu_{3}]^{1}, [\mu_{4}]^{9} \},$ where $\mu_1= 4(\omega_6+\omega_6^5)=4$, $\mu_2= 4(\omega_6\omega_3+\omega_6^5\omega_3^2)=-8$, $\mu_3= 4(\omega_6\omega_3^2+\omega_6^5\omega_3)=4$ and $\mu_4=0$.  
\end{ex}

\begin{table}
\begin{center}
\begin{tabular}{ c c c c c }
	\hline
	& $I$ & $\Cl((1,2)(3,4))$ &$\Cl((1,2,3))$ & $\Cl((1,3,2))$ \\
	\hline
	$\chi_1$   & $1$    &$~~1$  &$1$ & $1$\\
	$\chi_2$ &  $1$  & $~~1$  &$\omega_3$ &$\omega_3^2$\\
	$\chi_3$ &$1$ & $~~1$  &$\omega_3^2$&$\omega_3$\\
	$\chi_4$ &$3$ & $-1$  &$0$ &$0$\\
	\hline
\end{tabular}
\caption{Character table of $A_4$}\label{table1Ch5}
\end{center}
\end{table}

\section{HS-integral normal mixed Cayley graphs}\label{hs-mixed-integral}

In this section, we extend Theorem~\ref{NorOriCayGraphIntegCh5} to normal mixed Cayley graphs.

\begin{lema}\label{CharaNewIntegSumCh5}  Let $S$ be a skew-symmetric subset of a finite group $\Gamma$ and $\Irr(\Gamma)=\{ \chi_1,\ldots,\chi_h \}$. Let $S$ be expressible as a union of some conjugacy classes of $\Gamma$ and $t(\neq 0) \in \mathbb{Q}$. If $$\frac{1}{\chi_j(\mathbf 1)}\sum\limits_{s\in S} \i t \sqrt{3} \left(\chi_j(s)- \chi_j(s^{-1})\right)$$ is an integer for each $j \in \{ 1,\ldots , h \}$, then $S \in \mathbb{E}(\Gamma)$.
\end{lema}
\begin{proof}
	Let $x=\sum\limits_{g\in \Gamma} c_g g \in \mathbb{Q}(\omega_3)\Gamma$, where
	$$c_g= \left\{ \begin{array}{cl}
		\i t\sqrt{3} & \mbox{if }  g\in S \\
		-\i t\sqrt{3} & \mbox{if } g\in S^{-1}\\ 
		0 &  \mbox{otherwise}. 
	\end{array}\right.	$$ 
Note that  $\frac{\chi_j(x)}{\chi_j(\mathbf 1)}= \frac{1}{\chi_j(\mathbf 1)}\sum\limits_{s\in S} \i t\sqrt{3} \left( \chi_j(s)- \chi_j(s^{-1})\right)$. Assume that $\frac{\chi_j(x)}{\chi_j(\mathbf 1)}$ is an integer for each \linebreak[4] $j \in \{ 1,\ldots , h \}$. Therefore,  all the three conditions of Theorem \ref{MainTheoremIntCharaCh5} are satisfied for $x$. Using the fact that $g \sim g^{-1}$, and conditions (ii) and (iii) of Theorem \ref{MainTheoremIntCharaCh5}, we get $\Im \left(\sum\limits_{s\in \Cl(g)}c_s \right)=0$ for all $g \in \Gamma\setminus \Gamma (3)$, and so we must have $S \cup S^{-1} \subseteq \Gamma(3)$. Again, let  $g_1 \in S$, $g_2 \in \Gamma(3)$ and $ g_1 \simeq g_2$. The first condition of Theorem \ref{MainTheoremIntCharaCh5} gives $$\sum\limits_{s\in \Cl(g_1)}c_s=\sum\limits_{s\in \Cl(g_2)}c_s.$$ Note that $\sum\limits_{s\in \Cl(g_1)}c_s=\i t\sqrt{3}|\Cl(g_1)|$. Therefore $\sum\limits_{s\in \Cl(g_2)}c_s=\i t\sqrt{3}|\Cl(g_1)|$, and so $g_2\in S$. Thus $g_1 \in S$ implies $\langle\!\langle g_1 \rangle\!\rangle \subseteq S$. Hence $S\in \mathbb{E}(\Gamma)$. 
\end{proof}

In ~\cite{kadyan2021hsIntegralAbelian}, the authers proved that if $\Gamma$ is an abelian group, then $\langle\!\langle x \rangle\!\rangle \cup \langle\!\langle x^{-1} \rangle\!\rangle =[x]$ for each $x \in \Gamma(3)$. Note that this result and its proof also hold good for non-abelian group. In the subsequent discussion, we use this fact for non-abelian group. 

\begin{lema}\label{Sqrt3NecessIntSumCh5}  Let $S$ be a skew-symmetric subset of a finite group $\Gamma$ and $\Irr(\Gamma)=\{ \chi_1,\ldots,\chi_h \}$. Let $S$ be expressible as a union of some conjugacy classes of $\Gamma$ and $t(\neq 0) \in \mathbb{Q}$. If $$ \frac{1}{\chi_j(\mathbf 1)}\sum\limits_{s\in S} \i t\sqrt{3} \left(\chi_j(s)- \chi_j(s^{-1})\right)$$ is an integer for each $j \in \{ 1,\ldots , h \}$,  then $ \frac{1}{\chi_j(\mathbf 1)}\sum\limits_{s\in S\cup S^{-1}} \chi_j(s) $ is also an integer for each $j \in \{ 1,\ldots , h \}$.
\end{lema}
\begin{proof}
Assume that $\frac{1}{\chi_j(\mathbf 1)}\sum\limits_{s\in S} \i t\sqrt{3} \left(\chi_j(s)- \chi_j(s^{-1})\right)$ is an integer for each $j \in \{ 1,\ldots , h \}$. By Lemma \ref{CharaNewIntegSumCh5} we have $S \in \mathbb{E}(\Gamma)$, and so $S=\langle\!\langle x_1 \rangle\!\rangle\cup\cdots \cup \langle\!\langle x_k \rangle\!\rangle$ for some $x_1,\ldots,x_k\in \Gamma(3)$. Therefore, we get 
\[S \cup S^{-1}=\left(\langle\!\langle x_1 \rangle\!\rangle\cup\cdots \cup \langle\!\langle x_k \rangle\!\rangle\right)\cup \left(\langle\!\langle x_1^{-1} \rangle\!\rangle\cup\cdots \cup \langle\!\langle x_k^{-1} \rangle\!\rangle\right) =[x_1] \cup \cdots\cup[x_k] \in \mathbb{B}(\Gamma).\]
Thus by Theorem~\ref{NorMixCayGraphInteg}, $\Cay(\Gamma, S \cup S^{-1})$ is integral, that is, $ \frac{1}{\chi_j(\mathbf 1)}\sum\limits_{s\in S\cup S^{-1}} \chi_j(s) $ is an integer for each  $j \in \{ 1,\ldots , h \}$. 
\end{proof}

In the next result, we use the fact that the HS-eigenvalues of a mixed Cayley graph are algebraic integers. See Theorem 2.6 of \cite{li2022hermitian} for details.

\begin{lema}\label{SeperatIntegMixedGraphCh5}
If $\Gamma$ is a finite group, then the normal mixed Cayley graph $\Cay(\Gamma,S)$ is HS-integral if and only if  $\Cay(\Gamma,{S\setminus \overline{S}})$ is integral (or HS-integral) and $\Cay(\Gamma, {\overline{S}})$ is HS-integral.
\end{lema}
\begin{proof}
	Let $\Irr(\Gamma)=\{ \chi_1,\cdots,\chi_h \}$. By Lemma~\ref{EigNorCayMix}, the HS-spectrum of the normal mixed Cayley graph $\Cay(\Gamma,S)$ is $\{ [\gamma_{1}]^{d_1^2},\ldots, [\gamma_{h}]^{d_h^2} \},$ where $\gamma_j=\lambda_j + \mu_j$, $$\lambda_j= \frac{1}{\chi_j(\mathbf 1)} \sum_{s \in S \setminus \overline{S}} \chi_j(s),\hspace{0.2cm} \mu_j=\frac{1}{\chi_j(\mathbf 1)}  \sum_{s \in \overline{S}}(\omega_6 \chi_j(s)+\omega_6^5 \chi_j(s^{-1})),$$ and $d_j=\chi_j(\mathbf 1)$ for each $j\in \{1,\ldots,h\}$.
	Note that $\{ [\lambda_{1}]^{d_1^2},\ldots, [\lambda_{h}]^{d_h^2} \}$ is the spectrum of $\Cay(\Gamma, S\setminus \overline{S})$ and $\{ [\mu_{1}]^{d_1^2},\ldots, [\mu_{h}]^{d_h^2} \}$ is the HS-spectrum of $\Cay(\Gamma,\overline{S})$.
	
	Assume that the mixed Cayley graph $\Cay(\Gamma,S)$ is HS-integral. Let $j \in \{ 1,\ldots,h\}$. By Lemma~\ref{newLemmaConjugateChara}, there exists $k \in \{1,\ldots,h\}$ such that $\chi_{k}=\overline{\chi}_{j}$. Therefore, $\chi_j(\mathbf 1)=\chi_k( \mathbf 1)$ and $$\lambda_j = \frac{1}{\chi_j(\mathbf 1)} \sum_{s \in S \setminus \overline{S}} \chi_j(s^{-1})=  \frac{1}{\chi_j(\mathbf 1)} \sum_{s \in S \setminus \overline{S}} \overline{\chi_j(s)} = \frac{1}{\chi_k(\mathbf 1)} \sum_{s \in S \setminus \overline{S}} \chi_k(s) =\lambda_k.$$
Now we have
	\begin{align*}
	\gamma_j - \gamma_k &= \frac{1}{\chi_j(\mathbf 1)}  \sum_{s \in \overline{S}}\left( \omega_6 \chi_j(s)+\omega_6^5 \chi_j(s^{-1})\right) - \frac{1}{\chi_k(\mathbf 1)}  \sum_{s \in \overline{S}}\left(\omega_6 \chi_k(s)+\omega_6^5 \chi_k(s^{-1})\right)\\
	&=  \frac{1}{\chi_j(\mathbf 1)}  \sum_{s \in \overline{S}}\left(\omega_6 \chi_j(s)+\omega_6^5 \chi_j(s^{-1})\right) - \frac{1}{\chi_j(\mathbf 1)}  \sum_{s \in \overline{S}}\left(\omega_6 \overline{\chi_{j}(s)}+\omega_6^5 \overline{\chi_{j}(s^{-1})}\right)\\
	&=  \frac{1}{\chi_j(\mathbf 1)}  \sum_{s \in \overline{S}}\left(\omega_6 \chi_j(s)+\omega_6^5 \chi_j(s^{-1})\right) - \frac{1}{\chi_j(\mathbf 1)}  \sum_{s \in \overline{S}}\left(\omega_6 \chi_{j}(s^{-1})+\omega_6^5 \chi_{j}(s)\right)\\
	&=  \frac{1}{\chi_j(\mathbf 1)}  \sum_{s \in \overline{S}} \left( (\omega_6-\omega_6^5) \chi_j(s)+(\omega_6^5- \omega_6) \chi_j(s^{-1}) \right)\\
	&=  \frac{1}{\chi_j(\mathbf 1)}  \sum_{s \in \overline{S}} \i\sqrt{3} \left( \chi_j(s) - \chi_j(s^{-1})\right).
	\end{align*}
By assumption $\gamma_{j}, \gamma_{k} \in \mathbb{Z}$, and so $ \frac{1}{\chi_j(\mathbf 1)}  \sum\limits_{s \in \overline{S}} \i\sqrt{3} \left(\chi_j(s) - \chi_j(s^{-1})\right) \in \mathbb{Z}$ for each  $j \in \{ 1,\ldots , h \}$. Therefore by Lemma \ref{Sqrt3NecessIntSumCh5}, we get $\frac{1}{\chi_j(\mathbf 1)}\sum\limits_{ s \in \overline{S} \cup \overline{S}^{-1}} \chi_j(s) \in \mathbb{Z}$ for each $j \in \{ 1,\ldots , h \}$. Since 
	\begin{align*}
		\mu_{j}= \frac{1}{2\chi_j(\mathbf 1)} \sum\limits_{ s \in \overline{S} \cup \overline{S}^{-1}}\chi_j(s) + \frac{1}{2\chi_j(\mathbf 1)} \sum\limits_{s\in \overline{S}} \i\sqrt{3} \left(\chi_j(s) -\chi_j(s^{-1})\right),
	\end{align*} 
$\mu_j$ is a rational algebraic integer, and hence it is an integer for each $j \in \{ 1,\ldots , h \}$. Thus $\Cay(\Gamma,\overline{S})$ is HS-integral. Now we have $\gamma_{j}, \mu_{j} \in \mathbb{Z}$, and so $\lambda_{j} = \gamma_{j} -\mu_{j} \in \mathbb{Z}$ for each  $j \in \{ 1,\ldots , h \}$. Hence $\Cay(\Gamma,S\setminus \overline{S})$ is also integral.
	
Conversely, assume that $\Cay(\Gamma,S\setminus \overline{S})$ is integral and $\Cay(\Gamma, \overline{S})$ is HS-integral. Then Lemma~\ref{EigNorCayMix} implies that $\Cay(\Gamma,S)$ is HS-integral.
\end{proof}

\begin{theorem}\label{MixedHSintegralChara}
Let $\Gamma$ be a finite group and $\Cay(\Gamma, S)$ be a normal mixed Cayley graph. Then $\Cay(\Gamma, S)$ is HS-integral if and only if $S\setminus \overline{S} \in \mathbb{B}(\Gamma)$ and $\overline{S} \in \mathbb{E}(\Gamma)$.
\end{theorem}
\begin{proof}
By Lemma \ref{SeperatIntegMixedGraphCh5}, $\Cay(\Gamma, S)$ is HS-integral if and only if $\Cay(\Gamma, S \setminus \overline{S})$ is integral and $\Cay(\Gamma, \overline{S})$ is HS-integral. Now the proof follows from Theorem~\ref{NorMixCayGraphInteg} and Theorem~\ref{NorOriCayGraphIntegCh5}.
\end{proof}

\begin{figure}[ht]
\centering
\tikzset{every picture/.style={line width=0.75pt}} 

\begin{tikzpicture}[x=0.2pt,y=0.2pt,yscale=-1,xscale=1]

\draw    (841.32,693.15) -- (509.58,109.42) ;
\draw    (699.58,840.42) -- (509.58,109.42) ;
\draw    (510.58,890.42) -- (509.58,109.42) ;
\draw    (318.58,843.42) -- (509.58,109.42) ;
\draw    (159.32,696.15) -- (487.6,100.86) ;
\draw    (105.58,488.42) -- (487.6,100.86) ;
\draw    (167.58,285.42) -- (487.6,100.86) ;
\draw    (841.32,693.15) -- (694.32,164.58) ;
\draw    (699.58,840.42) -- (694.32,164.58) ;
\draw    (510.58,890.42) -- (694.32,164.58) ;
\draw    (318.58,843.42) -- (694.32,164.58) ;
\draw    (319.58,142.42) -- (681.6,151.86) ;
\draw    (167.58,285.42) -- (681.6,151.86) ;
\draw    (105.58,488.42) -- (681.6,151.86) ;
\draw    (159.32,696.15) -- (681.6,151.86) ;
\draw  [fill={rgb, 255:red, 0; green, 0; blue, 0 }  ,fill opacity=1 ] (485.48,222.16) -- (476.84,235.86) -- (476.57,219.67) -- (478.93,228.39) -- cycle ;
\draw    (841.32,693.15) -- (842.32,308.58) ;
\draw    (699.58,840.42) -- (842.32,308.58) ;
\draw    (510.58,890.42) -- (842.32,308.58) ;
\draw    (318.58,843.42) -- (842.32,308.58) ;
\draw    (699.58,840.42) -- (892.58,509.42) ;
\draw    (510.58,890.42) -- (892.58,509.42) ;
\draw    (318.58,843.42) -- (892.58,509.42) ;
\draw    (319.58,142.42) -- (829.6,295.86) ;
\draw    (167.58,285.42) -- (829.6,295.86) ;
\draw    (105.58,488.42) -- (829.6,295.86) ;
\draw    (159.32,696.15) -- (829.6,295.86) ;
\draw    (319.58,142.42) -- (888.58,493.42) ;
\draw    (167.58,285.42) -- (888.58,493.42) ;
\draw    (105.58,488.42) -- (888.58,493.42) ;
\draw    (159.32,696.15) -- (888.58,493.42) ;
\draw    (314.58,158.42) -- (831.58,701.42) ;
\draw    (163.58,303.42) -- (831.58,701.42) ;
\draw    (112.58,506.42) -- (831.58,701.42) ;
\draw    (170.58,715.42) -- (831.58,701.42) ;
\draw    (314.58,158.42) -- (682.58,841.42) ;
\draw    (163.58,303.42) -- (682.58,841.42) ;
\draw    (112.58,506.42) -- (682.58,841.42) ;
\draw    (170.58,715.42) -- (682.58,841.42) ;
\draw    (314.58,158.42) -- (493.58,890.42) ;
\draw    (163.58,303.42) -- (493.58,890.42) ;
\draw    (112.58,506.42) -- (493.58,890.42) ;
\draw    (170.58,715.42) -- (493.58,890.42) ;
\draw    (314.58,158.42) -- (302.58,840.42) ;
\draw    (163.58,303.42) -- (302.58,840.42) ;
\draw    (112.58,506.42) -- (302.58,840.42) ;
\draw  [fill={rgb, 255:red, 0; green, 0; blue, 0 }  ,fill opacity=1 ] (514.68,227.7) -- (509.81,243.15) -- (505.43,227.56) -- (509.93,235.39) -- cycle ;
\draw  [fill={rgb, 255:red, 0; green, 0; blue, 0 }  ,fill opacity=1 ] (540.93,213.58) -- (541.34,229.77) -- (532.13,216.45) -- (538.93,222.39) -- cycle ;
\draw  [fill={rgb, 255:red, 0; green, 0; blue, 0 }  ,fill opacity=1 ] (562.15,193.36) -- (565.75,209.15) -- (554.09,197.91) -- (561.93,202.39) -- cycle ;
\draw  [fill={rgb, 255:red, 0; green, 0; blue, 0 }  ,fill opacity=1 ] (644.72,263.8) -- (633.21,275.2) -- (636.6,259.36) -- (636.93,268.39) -- cycle ;
\draw  [fill={rgb, 255:red, 0; green, 0; blue, 0 }  ,fill opacity=1 ] (670.32,278.99) -- (662.04,292.91) -- (661.34,276.73) -- (663.93,285.39) -- cycle ;
\draw  [fill={rgb, 255:red, 0; green, 0; blue, 0 }  ,fill opacity=1 ] (700.04,272.93) -- (694.45,288.13) -- (690.8,272.35) -- (694.93,280.39) -- cycle ;
\draw  [fill={rgb, 255:red, 0; green, 0; blue, 0 }  ,fill opacity=1 ] (728.9,270.85) -- (727.5,286.99) -- (719.83,272.73) -- (725.93,279.39) -- cycle ;
\draw  [fill={rgb, 255:red, 0; green, 0; blue, 0 }  ,fill opacity=1 ] (774.43,384.31) -- (761.02,393.39) -- (767.27,378.45) -- (765.93,387.39) -- cycle ;
\draw  [fill={rgb, 255:red, 0; green, 0; blue, 0 }  ,fill opacity=1 ] (791.1,408.53) -- (778.61,418.83) -- (783.42,403.37) -- (782.93,412.39) -- cycle ;
\draw  [fill={rgb, 255:red, 0; green, 0; blue, 0 }  ,fill opacity=1 ] (815.69,426.39) -- (806.59,439.79) -- (806.86,423.59) -- (808.93,432.39) -- cycle ;
\draw  [fill={rgb, 255:red, 0; green, 0; blue, 0 }  ,fill opacity=1 ] (846.56,428.63) -- (841.93,444.15) -- (837.31,428.63) -- (841.93,436.39) -- cycle ;
\draw  [fill={rgb, 255:red, 0; green, 0; blue, 0 }  ,fill opacity=1 ] (823.97,554.5) -- (808.22,558.28) -- (819.33,546.49) -- (814.93,554.39) -- cycle ;
\draw  [fill={rgb, 255:red, 0; green, 0; blue, 0 }  ,fill opacity=1 ] (829.85,578.92) -- (815,585.39) -- (823.88,571.85) -- (820.93,580.39) -- cycle ;
\draw  [fill={rgb, 255:red, 0; green, 0; blue, 0 }  ,fill opacity=1 ] (839.01,610.63) -- (826.76,621.21) -- (831.22,605.65) -- (830.93,614.68) -- cycle ;
\draw  [fill={rgb, 255:red, 0; green, 0; blue, 0 }  ,fill opacity=1 ] (894.59,593.57) -- (885.71,607.11) -- (885.73,590.91) -- (887.93,599.68) -- cycle ;
\draw  [fill={rgb, 255:red, 0; green, 0; blue, 0 }  ,fill opacity=1 ] (724.04,707.66) -- (708.2,704.3) -- (723.29,698.43) -- (715.93,703.68) -- cycle ;
\draw  [fill={rgb, 255:red, 0; green, 0; blue, 0 }  ,fill opacity=1 ] (716.51,674.87) -- (702.36,666.99) -- (718.51,665.84) -- (709.93,668.68) -- cycle ;
\draw  [fill={rgb, 255:red, 0; green, 0; blue, 0 }  ,fill opacity=1 ] (723.57,642.43) -- (712.11,630.99) -- (727.96,634.29) -- (718.93,634.68) -- cycle ;
\draw  [fill={rgb, 255:red, 0; green, 0; blue, 0 }  ,fill opacity=1 ] (746.93,619.49) -- (739.58,605.06) -- (753.64,613.1) -- (744.93,610.68) -- cycle ;
\draw  [fill={rgb, 255:red, 0; green, 0; blue, 0 }  ,fill opacity=1 ] (621.56,737.7) -- (618.24,721.85) -- (629.7,733.3) -- (621.93,728.68) -- cycle ;
\draw  [fill={rgb, 255:red, 0; green, 0; blue, 0 }  ,fill opacity=1 ] (600.54,763.57) -- (593.84,748.82) -- (607.52,757.49) -- (598.93,754.68) -- cycle ;
\draw  [fill={rgb, 255:red, 0; green, 0; blue, 0 }  ,fill opacity=1 ] (588.07,791.71) -- (577.35,779.57) -- (592.97,783.86) -- (583.93,783.68) -- cycle ;
\draw  [fill={rgb, 255:red, 0; green, 0; blue, 0 }  ,fill opacity=1 ] (571.78,819.56) -- (558.59,810.17) -- (574.77,810.81) -- (565.93,812.68) -- cycle ;
\draw  [fill={rgb, 255:red, 0; green, 0; blue, 0 }  ,fill opacity=1 ] (462.83,782.46) -- (462.62,766.27) -- (471.67,779.7) -- (464.93,773.68) -- cycle ;
\draw  [fill={rgb, 255:red, 0; green, 0; blue, 0 }  ,fill opacity=1 ] (438.56,801.7) -- (435.24,785.85) -- (446.7,797.3) -- (438.93,792.68) -- cycle ;
\draw  [fill={rgb, 255:red, 0; green, 0; blue, 0 }  ,fill opacity=1 ] (411.98,815.48) -- (404.55,801.09) -- (418.65,809.06) -- (409.93,806.68) -- cycle ;
\draw  [fill={rgb, 255:red, 0; green, 0; blue, 0 }  ,fill opacity=1 ] (387.43,838.51) -- (376.17,826.87) -- (391.97,830.45) -- (382.93,830.68) -- cycle ;
\draw  [fill={rgb, 255:red, 0; green, 0; blue, 0 }  ,fill opacity=1 ] (300.5,735.55) -- (304.75,719.92) -- (309.75,735.32) -- (304.93,727.68) -- cycle ;
\draw  [fill={rgb, 255:red, 0; green, 0; blue, 0 }  ,fill opacity=1 ] (271.3,736.32) -- (272.07,720.14) -- (280.29,734.1) -- (273.93,727.68) -- cycle ;
\draw  [fill={rgb, 255:red, 0; green, 0; blue, 0 }  ,fill opacity=1 ] (247.31,751.7) -- (242.68,736.18) -- (255.06,746.63) -- (246.93,742.68) -- cycle ;
\draw  [fill={rgb, 255:red, 0; green, 0; blue, 0 }  ,fill opacity=1 ] (221.77,792.52) -- (214.69,777.96) -- (228.59,786.27) -- (219.93,783.68) -- cycle ;
\draw  [fill={rgb, 255:red, 0; green, 0; blue, 0 }  ,fill opacity=1 ] (211.79,590.58) -- (224.22,580.21) -- (219.5,595.7) -- (219.93,586.68) -- cycle ;
\draw  [fill={rgb, 255:red, 0; green, 0; blue, 0 }  ,fill opacity=1 ] (242.26,603.19) -- (256.23,595) -- (249.02,609.51) -- (250.93,600.68) -- cycle ;
\draw  [fill={rgb, 255:red, 0; green, 0; blue, 0 }  ,fill opacity=1 ] (267.91,626.21) -- (283.79,623.05) -- (272.24,634.39) -- (276.93,626.68) -- cycle ;
\draw  [fill={rgb, 255:red, 0; green, 0; blue, 0 }  ,fill opacity=1 ] (278.38,657.77) -- (294.52,659.05) -- (280.31,666.82) -- (286.93,660.68) -- cycle ;
\draw  [fill={rgb, 255:red, 0; green, 0; blue, 0 }  ,fill opacity=1 ] (192.23,395.08) -- (206.3,387.07) -- (198.91,401.48) -- (200.93,392.68) -- cycle ;
\draw  [fill={rgb, 255:red, 0; green, 0; blue, 0 }  ,fill opacity=1 ] (205.9,424.76) -- (221.56,420.64) -- (210.71,432.66) -- (214.93,424.68) -- cycle ;
\draw  [fill={rgb, 255:red, 0; green, 0; blue, 0 }  ,fill opacity=1 ] (208.19,456.39) -- (224.39,456.52) -- (210.77,465.28) -- (216.93,458.68) -- cycle ;
\draw  [fill={rgb, 255:red, 0; green, 0; blue, 0 }  ,fill opacity=1 ] (224.26,483.9) -- (239.69,488.82) -- (224.09,493.16) -- (231.93,488.68) -- cycle ;
\draw  [fill={rgb, 255:red, 0; green, 0; blue, 0 }  ,fill opacity=1 ] (280.94,213.81) -- (296.95,211.36) -- (284.9,222.18) -- (289.93,214.68) -- cycle ;
\draw  [fill={rgb, 255:red, 0; green, 0; blue, 0 }  ,fill opacity=1 ] (291.4,247.72) -- (307.53,249.1) -- (293.28,256.78) -- (299.93,250.68) -- cycle ;
\draw  [fill={rgb, 255:red, 0; green, 0; blue, 0 }  ,fill opacity=1 ] (283.27,282.9) -- (298.69,287.82) -- (283.09,292.15) -- (290.93,287.68) -- cycle ;
\draw  [fill={rgb, 255:red, 0; green, 0; blue, 0 }  ,fill opacity=1 ] (284.31,313.61) -- (297.2,323.41) -- (281.04,322.27) -- (289.93,320.68) -- cycle ;
\draw  [fill={rgb, 255:red, 0; green, 0; blue, 0 }  ,fill opacity=1 ] (391.21,111.34) -- (407.4,111.56) -- (393.73,120.24) -- (399.93,113.68) -- cycle ;
\draw  [fill={rgb, 255:red, 0; green, 0; blue, 0 }  ,fill opacity=1 ] (422.51,140.52) -- (437.68,146.21) -- (421.87,149.75) -- (429.93,145.68) -- cycle ;
\draw  [fill={rgb, 255:red, 0; green, 0; blue, 0 }  ,fill opacity=1 ] (397.39,160.54) -- (410.16,170.5) -- (394.02,169.16) -- (402.93,167.68) -- cycle ;
\draw  [fill={rgb, 255:red, 0; green, 0; blue, 0 }  ,fill opacity=1 ] (381.09,174.5) -- (391.37,187.02) -- (375.91,182.17) -- (384.93,182.68) -- cycle ;
\draw   (101.82,500.36) .. controls (101.82,279.73) and (280.69,100.86) .. (501.32,100.86) .. controls (721.96,100.86) and (900.82,279.73) .. (900.82,500.36) .. controls (900.82,721) and (721.96,899.86) .. (501.32,899.86) .. controls (280.69,899.86) and (101.82,721) .. (101.82,500.36) -- cycle ;
\draw  [color={rgb, 255:red, 0; green, 0; blue, 0 }  ,draw opacity=1 ][fill={rgb, 255:red, 255; green, 255; blue, 255 }  ,fill opacity=1 ][line width=0.75]  (145.6,292.86) .. controls (145.6,285.84) and (151.3,280.15) .. (158.32,280.15) .. controls (165.35,280.15) and (171.04,285.84) .. (171.04,292.86) .. controls (171.04,299.89) and (165.35,305.58) .. (158.32,305.58) .. controls (151.3,305.58) and (145.6,299.89) .. (145.6,292.86) -- cycle ;
\draw  [color={rgb, 255:red, 0; green, 0; blue, 0 }  ,draw opacity=1 ][fill={rgb, 255:red, 255; green, 255; blue, 255 }  ,fill opacity=1 ][line width=0.75]  (295.86,145.7) .. controls (295.86,138.67) and (301.56,132.98) .. (308.58,132.98) .. controls (315.61,132.98) and (321.3,138.67) .. (321.3,145.7) .. controls (321.3,152.72) and (315.61,158.42) .. (308.58,158.42) .. controls (301.56,158.42) and (295.86,152.72) .. (295.86,145.7) -- cycle ;
\draw  [color={rgb, 255:red, 0; green, 0; blue, 0 }  ,draw opacity=1 ][fill={rgb, 255:red, 255; green, 255; blue, 255 }  ,fill opacity=1 ][line width=0.75]  (487.6,100.86) .. controls (487.6,93.84) and (493.3,88.15) .. (500.32,88.15) .. controls (507.35,88.15) and (513.04,93.84) .. (513.04,100.86) .. controls (513.04,107.89) and (507.35,113.58) .. (500.32,113.58) .. controls (493.3,113.58) and (487.6,107.89) .. (487.6,100.86) -- cycle ;
\draw  [color={rgb, 255:red, 0; green, 0; blue, 0 }  ,draw opacity=1 ][fill={rgb, 255:red, 255; green, 255; blue, 255 }  ,fill opacity=1 ][line width=0.75]  (681.6,151.86) .. controls (681.6,144.84) and (687.3,139.15) .. (694.32,139.15) .. controls (701.35,139.15) and (707.04,144.84) .. (707.04,151.86) .. controls (707.04,158.89) and (701.35,164.58) .. (694.32,164.58) .. controls (687.3,164.58) and (681.6,158.89) .. (681.6,151.86) -- cycle ;
\draw  [color={rgb, 255:red, 0; green, 0; blue, 0 }  ,draw opacity=1 ][fill={rgb, 255:red, 255; green, 255; blue, 255 }  ,fill opacity=1 ][line width=0.75]  (829.6,295.86) .. controls (829.6,288.84) and (835.3,283.15) .. (842.32,283.15) .. controls (849.35,283.15) and (855.04,288.84) .. (855.04,295.86) .. controls (855.04,302.89) and (849.35,308.58) .. (842.32,308.58) .. controls (835.3,308.58) and (829.6,302.89) .. (829.6,295.86) -- cycle ;
\draw  [color={rgb, 255:red, 0; green, 0; blue, 0 }  ,draw opacity=1 ][fill={rgb, 255:red, 255; green, 255; blue, 255 }  ,fill opacity=1 ][line width=0.75]  (888.1,500.36) .. controls (888.1,493.34) and (893.8,487.65) .. (900.82,487.65) .. controls (907.85,487.65) and (913.54,493.34) .. (913.54,500.36) .. controls (913.54,507.39) and (907.85,513.08) .. (900.82,513.08) .. controls (893.8,513.08) and (888.1,507.39) .. (888.1,500.36) -- cycle ;
\draw  [color={rgb, 255:red, 0; green, 0; blue, 0 }  ,draw opacity=1 ][fill={rgb, 255:red, 255; green, 255; blue, 255 }  ,fill opacity=1 ][line width=0.75]  (830.58,704.42) .. controls (830.58,697.39) and (836.28,691.7) .. (843.3,691.7) .. controls (850.33,691.7) and (856.02,697.39) .. (856.02,704.42) .. controls (856.02,711.44) and (850.33,717.14) .. (843.3,717.14) .. controls (836.28,717.14) and (830.58,711.44) .. (830.58,704.42) -- cycle ;
\draw  [color={rgb, 255:red, 0; green, 0; blue, 0 }  ,draw opacity=1 ][fill={rgb, 255:red, 255; green, 255; blue, 255 }  ,fill opacity=1 ][line width=0.75]  (679.58,849.42) .. controls (679.58,842.39) and (685.28,836.7) .. (692.3,836.7) .. controls (699.33,836.7) and (705.02,842.39) .. (705.02,849.42) .. controls (705.02,856.44) and (699.33,862.14) .. (692.3,862.14) .. controls (685.28,862.14) and (679.58,856.44) .. (679.58,849.42) -- cycle ;
\draw  [color={rgb, 255:red, 0; green, 0; blue, 0 }  ,draw opacity=1 ][fill={rgb, 255:red, 255; green, 255; blue, 255 }  ,fill opacity=1 ][line width=0.75]  (488.6,899.86) .. controls (488.6,892.84) and (494.3,887.15) .. (501.32,887.15) .. controls (508.35,887.15) and (514.04,892.84) .. (514.04,899.86) .. controls (514.04,906.89) and (508.35,912.58) .. (501.32,912.58) .. controls (494.3,912.58) and (488.6,906.89) .. (488.6,899.86) -- cycle ;
\draw  [color={rgb, 255:red, 0; green, 0; blue, 0 }  ,draw opacity=1 ][fill={rgb, 255:red, 255; green, 255; blue, 255 }  ,fill opacity=1 ][line width=0.75]  (296.15,851.42) .. controls (296.15,844.39) and (301.84,838.7) .. (308.86,838.7) .. controls (315.89,838.7) and (321.58,844.39) .. (321.58,851.42) .. controls (321.58,858.44) and (315.89,864.14) .. (308.86,864.14) .. controls (301.84,864.14) and (296.15,858.44) .. (296.15,851.42) -- cycle ;
\draw  [color={rgb, 255:red, 0; green, 0; blue, 0 }  ,draw opacity=1 ][fill={rgb, 255:red, 255; green, 255; blue, 255 }  ,fill opacity=1 ][line width=0.75]  (146.6,708.86) .. controls (146.6,701.84) and (152.3,696.15) .. (159.32,696.15) .. controls (166.35,696.15) and (172.04,701.84) .. (172.04,708.86) .. controls (172.04,715.89) and (166.35,721.58) .. (159.32,721.58) .. controls (152.3,721.58) and (146.6,715.89) .. (146.6,708.86) -- cycle ;
\draw  [color={rgb, 255:red, 0; green, 0; blue, 0 }  ,draw opacity=1 ][fill={rgb, 255:red, 255; green, 255; blue, 255 }  ,fill opacity=1 ][line width=0.75]  (89.1,500.36) .. controls (89.1,493.34) and (94.8,487.65) .. (101.82,487.65) .. controls (108.85,487.65) and (114.54,493.34) .. (114.54,500.36) .. controls (114.54,507.39) and (108.85,513.08) .. (101.82,513.08) .. controls (94.8,513.08) and (89.1,507.39) .. (89.1,500.36) -- cycle ;
\draw    (829.6,295.86) -- (509.58,109.42) ;
\draw    (888.58,493.42) -- (509.58,109.42) ;
\draw    (888.58,493.42) -- (694.32,164.58) ;
\draw    (510.58,890.42) -- (831.58,701.42) ;
\draw    (318.58,843.42) -- (831.58,701.42) ;
\draw    (318.58,843.42) -- (682.58,841.42) ;
\draw    (159.32,696.15) -- (163.58,303.42) ;
\draw    (159.32,696.15) -- (314.58,158.42) ;
\draw    (105.58,488.42) -- (314.58,158.42) ;

\draw (490,30) node [anchor=north west][inner sep=0.75pt]    {$I$};
\draw (710,102) node [anchor=north west][inner sep=0.75pt]  {$( 1,2)( 3,4)$};
\draw (860,252) node [anchor=north west][inner sep=0.75pt]   {$( 1,3)( 2,4)$};
\draw (915,475) node [anchor=north west][inner sep=0.75pt]   {$( 1,4)( 2,3)$};
\draw (860,694) node [anchor=north west][inner sep=0.75pt]    {$( 1,2,3)$};
\draw (702,850) node [anchor=north west][inner sep=0.75pt]   {$( 4,2,1)$};
\draw (415,917.57) node [anchor=north west][inner sep=0.75pt]    {$( 2,4,3)$};
\draw (110,850) node [anchor=north west][inner sep=0.75pt]   {$( 3,4,1)$};
\draw (-45,694) node [anchor=north west][inner sep=0.75pt]  {$( 1,3,2)$};
\draw (-100,475) node [anchor=north west][inner sep=0.75pt]    {$( 4,1,2)$};
\draw (-45,252) node [anchor=north west][inner sep=0.75pt]  {$( 2,3,4)$};
\draw (110,102) node [anchor=north west][inner sep=0.75pt]  {$( 3,1,4)$};
\end{tikzpicture}
\caption{The mixed graph $\Cay(A_4,S)$}\label{Fig2}
\end{figure}
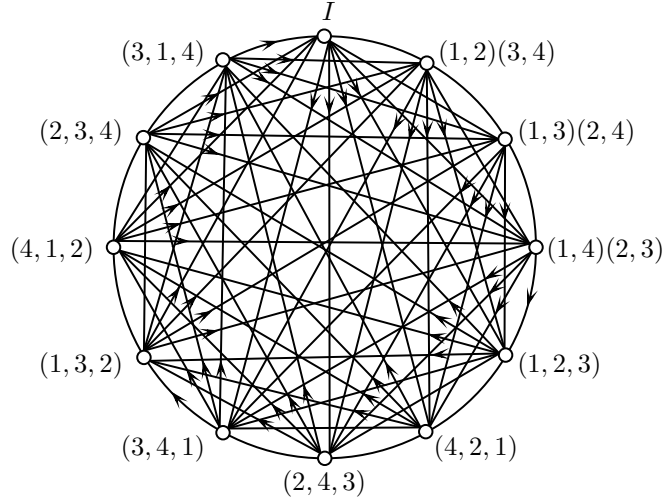

We give the following example to illustrate Theorem~\ref{MixedHSintegralChara}.

\begin{ex}\label{exampleMixed} \normalfont Consider $$S=\{ (1,2)(3,4),(1,3)(2,4), (1,4)(2,3), (1,2,3), (4,2,1), (2,4,3), (3,4,1) \}$$ in the alternating group $A_4$. The normal mixed Cayley graph $\Cay(A_4, S)$ is shown in Figure~\ref{Fig2}. We find that $$\overline{S}= \langle\!\langle (1,2,3) \rangle\!\rangle \cup \langle\!\langle (4,2,1)\rangle\!\rangle \cup \langle\!\langle (2,4,3) \rangle\!\rangle \cup \langle\!\langle (3,4,1)\rangle\!\rangle= \Cl((1,2,3)) \in \mathbb{E}(\Gamma)$$ and  $$S \setminus \overline{S}= [(1,2)(3,4)] \cup [(1,3)(2,4)] \cup [(1,4)(2,3)]=\Cl((1,2)(3,4)) \in \mathbb{B}(\Gamma).$$  Using Theorem~\ref{MixedHSintegralChara}, $\Cay(A_4, S)$ is HS-integral. The character table of $A_4$ is given in  Table~\ref{table1Ch5}. Further, using Lemma~\ref{EigNorCayMix}, the HS-spectrum of  $\Cay(A_4, S)$ is obtained as  $\{ [\gamma_{1}]^{1},[\gamma_{2}]^{1},[\gamma_{3}]^{1}, [\gamma_{4}]^{9} \},$ where \linebreak[4]$\gamma_1=3+ 4(\omega_6+\omega_6^5)=7$, $\gamma_2= 3+4(\omega_6\omega_3+\omega_6^5\omega_3^2)=-5$, $\gamma_3=3+ 4(\omega_6\omega_3^2+\omega_6^5\omega_3)=7$ and $\gamma_4=-1$. 
\end{ex}

\section{Eisenstein integral normal mixed Cayley graphs}\label{hs-mixed-Eisenstein-integral}\label{Sec5}
Assume that $S$ is a union of some conjugacy classes of a finite group $\Gamma$, $\mathbf 1 \not\in S$ and  $\Irr(\Gamma)=\{ \chi_1,\ldots,\chi_h\}$. Using the function $f \colon \Gamma \rightarrow \{0,1\}$ defined by 
\[f(s)= \left\{ \begin{array}{rl}
		1 & \mbox{if } s\in S \\
		0 &   \mbox{otherwise} 
	\end{array}\right.\]
in Theorem~\ref{EigNorColCayMix1}, we find that $\frac{1}{\chi_j(\mathbf 1)}\sum\limits_{s \in S} \chi_{j}(s)$ is an eigenvalue of the normal mixed Cayley graph $\Cay(\Gamma, S)$ for each $j \in \{ 1,\ldots , h \}$. Indeed, all the eigenvalues of $\Cay(\Gamma, S)$ are of this form.

For each $j \in \{ 1,\ldots,h\}$, define
$$f_{j}(S) :=\frac{1}{\chi_j(\mathbf 1)} \sum_{s \in S \setminus \overline{S}} \chi_{j}(s)  \hspace{0.5cm}\textnormal{and}\hspace{0.5cm} g_{j}(S):= \frac{1}{\chi_j(\mathbf 1)} \sum_{s \in \overline{S}}(\omega \chi_{j}(s) + \overline{\omega}\chi_{j}(s^{-1})),$$ where $\omega=\frac{1}{2} - \frac{\i \sqrt{3}}{6}$. Let $j \in \{ 1,\ldots,h\}$. By Lemma~\ref{newLemmaConjugateChara}, there exists $k \in \{ 1,\ldots,h\}$ such that $\chi_k=\overline{\chi}_j$. Note that 
\begin{align*}
g_{j}(S) + \omega_3( g_{j}(S) - g_{k}(S) ) =& (1+\omega_3)g_{j}(S) - \omega_3 g_{k}(S)\nonumber\\
=& \frac{1+ \i\sqrt{3} }{2 \chi_j(\mathbf 1)} \sum_{s \in \overline{S}}\left[\left(\frac{1}{2} - \frac{\i \sqrt{3}}{6}\right) \chi_{j}(s) + \left(\frac{1}{2} + \frac{\i \sqrt{3}}{6}\right)\chi_{j}(s^{-1})\right] \nonumber\\
+&  \frac{1- \i\sqrt{3} }{2 \chi_j(\mathbf 1)} \sum_{s \in \overline{S}}\left[ \left(\frac{1}{2} - \frac{\i \sqrt{3}}{6}\right) \chi_{k}(s) + \left(\frac{1}{2} + \frac{\i \sqrt{3}}{6}\right)\chi_{k}(s^{-1})\right]\nonumber\\
=& \frac{1+ \i\sqrt{3} }{2 \chi_j(\mathbf 1)} \sum_{s \in \overline{S}}\left[\left(\frac{1}{2} - \frac{\i \sqrt{3}}{6}\right) \chi_{j}(s) + \left(\frac{1}{2} + \frac{\i \sqrt{3}}{6}\right)\chi_{j}(s^{-1})\right] \nonumber\\
+&  \frac{1- \i\sqrt{3} }{2 \chi_j(\mathbf 1)} \sum_{s \in \overline{S}}\left[ \left(\frac{1}{2} - \frac{\i \sqrt{3}}{6}\right) \chi_{j}(s^{-1}) + \left(\frac{1}{2} + \frac{\i \sqrt{3}}{6}\right)\chi_{j}(s)\right]\nonumber\\
=& \frac{1}{\chi_j(\mathbf 1)} \sum_{s \in \overline{S}} \chi_{j}(s).
\end{align*}
Therefore
\begin{equation}\label{eqEisenIntPartCh5}
\begin{split}
\frac{1}{\chi_j(\mathbf 1)} \sum_{s \in S} \chi_{j}(s)= f_{j}(S)+ g_{j}(S) + \omega_3( g_{j}(S) - g_{k}(S) ).
\end{split}
\end{equation}
Note that if $\chi_k=\overline{\chi}_j$, then $f_{j}(S) = f_{k}(S)$ and $g_{j}(S)-g_{k}(S)= \left[ f_{j}(S)+g_{j}(S)\right]-\left[ f_{k}(S)+g_{k}(S)\right]$. Therefore if $f_{j}(S) + g_{j}(S)$ is an integer for each $j \in \{ 1,\ldots , h \}$, then $g_{j}(S)-g_{k}(S)$ is also an integer for each $j \in \{ 1,\ldots , h \}$. Hence the normal mixed Cayley graph $\Cay(\Gamma,S)$ is Eisenstein integral if and only if $f_{j}(S) + g_{j}(S)$ is an integer for each $j \in \{ 1,\ldots , h \}$.

\begin{lema}\label{CharaEisensteinIntegralCh5}
If $\Gamma$ is a finite group, then the normal mixed Cayley graph $\Cay(\Gamma,S)$ is Eisenstein integral if and only if $2 f_{j}(S)$ and $2 g_{j}(S)$ are integers of the same parity for each $j \in \{ 1,\ldots , h \}$.
\end{lema}
\begin{proof} 
Assume that the normal mixed Cayley graph $\Cay(\Gamma,S)$ is Eisenstein integral. Then $f_{j}(S) + g_{j}(S)$ and $g_{j}(S)-g_{k}(S)$ are integers for each $j \in \{ 1,\ldots , h \}$, where $\chi_k=\overline{\chi}_j$. Note that $$g_{j}(S)-g_{k}(S)= \frac{1}{\chi_j(\mathbf 1)} \sum\limits_{s\in \overline{S}} \frac{-\i \sqrt{3}}{3}( \chi_{j}(s)-  \chi_{j}(s^{-1})).$$ Therefore by Lemma~\ref{Sqrt3NecessIntSumCh5}, $\frac{1}{\chi_j(\mathbf 1)} \sum\limits_{s\in \overline{S}\cup \overline{S}^{-1}} \chi_{j}(s) \in \mathbb{Z}$. Using 
$$2 g_{j}(S)= \frac{1}{\chi_j(\mathbf 1)} \sum\limits_{s\in \overline{S}\cup \overline{S}^{-1}} \chi_{j}(s)- \frac{1}{\chi_j(\mathbf 1)} \sum\limits_{s\in \overline{S}}\frac{\i\sqrt{3}}{3} (\chi_{j}(s)- \chi_{j}(s^{-1})),$$ we find that $2 g_{j}(S)$ is an integer. Since $2 f_{j}(S)=2(f_{j}(S)+g_{j}(S))-2g_{j}(S)$, we see that $2 f_{j}(S)$ is also an integer of the same parity with $2g_{j}(S)$.

Conversely, assume that $2 f_{j}(S)$ and $2 g_{j}(S)$ are integers of the same parity for each $j \in \{ 1,\ldots , h \}$. Then $f_{j}(S) + g_{j}(S)$ is an integer for each $j \in \{ 1,\ldots , h \}$. Hence the normal mixed Cayley graph $\Cay(\Gamma,S)$ is Eisenstein integral.
\end{proof}

\begin{lema}\label{CharaEisensteinIntegral1Ch5}
The normal mixed Cayley graph $\Cay(\Gamma,S)$ is Eisenstein integral if and only if $f_{j}(S)$ and $g_{j}(S)$ are integers for each $j \in \{ 1,\ldots , h \}$.
\end{lema}
\begin{proof}
Let $j \in \{ 1,\ldots , h \}$. Due to Lemma~\ref{CharaEisensteinIntegralCh5}, it is enough to prove that $2 f_{j}(S)$ and $2 g_{j}(S)$ are integers of the same parity if and only if $f_{j}(S)$ and $g_{j}(S)$ are integers. If $f_{j}(S)$ and $g_{j}(S)$ are integers, then clearly $2 f_{j}(S)$ and $2 g_j(S)$ are even integers. Conversely, assume that $2 f_{j}(S)$ and $2 g_{j}(S)$ are integers of the same parity. Since $f_{j}(S)$ is an algebraic integer, the integrality of $2 f_{j}(S)$ implies that $f_{j}(S)$ is an integer. Thus $2 f_{j}(S)$ is an even integer, and so by assumption $2 g_{j}(S)$ is also an even integer. Hence $g_{j}(S)$ is an integer.
\end{proof}

\begin{theorem}\label{MinCharacEisensteinIntegCh5}
Let $\Gamma$ be a finite group. If the normal mixed Cayley graph $\Cay(\Gamma,S)$ is Eisenstein integral, then $\Cay(\Gamma,S)$ is HS-integral.
\end{theorem}
\begin{proof}
Assume that $\Cay(\Gamma,S)$ is Eisenstein integral. By Lemma~\ref{CharaEisensteinIntegral1Ch5}, we find that $f_{j}(S)$ and $g_{j}(S)$ are integers for each $j \in \{ 1,\ldots , h \}$. Note that $f_{j}(S)$ is an eigenvalue of the normal simple Cayley graph $\Cay(\Gamma,S\setminus \overline{S})$. By Theorem~\ref{NorMixCayGraphInteg}, $f_{j}(S)$ is an integer for each $j \in \{ 1,\ldots , h \}$ if and only if $S\setminus\overline{S} \in \mathbb{B}(\Gamma)$. Further, $$ \frac{1}{\chi_j(\mathbf 1)} \sum\limits_{s\in \overline{S}} \frac{-\i\sqrt{3}}{3}( \chi_{j}(s)- \chi_{j}(s^{-1}))=g_{j}(S)-g_{k}(S),$$ and that $g_{j}(S)-g_{k}(S)$ is an integer for each $j \in \{ 1,\ldots , h \}$, where $\chi_k=\overline{\chi}_j$. Using Lemma~\ref{CharaNewIntegSumCh5}, we see that $\overline{S} \in \mathbb{E}(\Gamma)$. Thus by Theorem~\ref{MixedHSintegralChara}, $\Cay(\Gamma,S)$ is HS-integral.
\end{proof}

\begin{lema}\label{NewLemmaEquivaTrans1Ch5} Let $ x\in \Gamma$ and $\ord(x)=3^{t}m$. If  $m\not\equiv 0 \Mod 3$, then the following assertions hold.
\begin{enumerate}[label=(\roman*)]
\item If $t=1$, then $[x]=x^m[x^3] \cup x^{2m} [x^3]$.
\item If $t=1$, then $$\langle\!\langle x \rangle\!\rangle = \left\{ \begin{array}{ll}
			x^m[x^3]     & \mbox{if } m \equiv 1 \Mod 3  \\
			x^{2m}[x^3] & \mbox{if } m \equiv 2 \Mod 3.
		\end{array}\right. $$
\item If $t \geq 2$, then $$[x] = \left\{ \begin{array}{ll}
			x^m[x^3] \cup x^{2m}[x^3] \cup x^{4m}\langle\!\langle x^{-3}\rangle\!\rangle \cup x^{5m}\langle\!\langle x^{-3}\rangle\!\rangle  & \mbox{if } m \equiv 1 \Mod 3  \\
			x^m[x^3] \cup x^{2m}[x^3] \cup x^{4m}\langle\!\langle x^{3}\rangle\!\rangle \cup x^{5m}\langle\!\langle x^{3}\rangle\!\rangle & \mbox{if } m \equiv 2 \Mod 3.
		\end{array}\right. $$
\item If $t \geq 2$, then $$[x] = \left\{ \begin{array}{ll}
			x^{7m}[x^3] \cup x^{8m}[x^3] \cup x^{4m}\langle\!\langle x^{3}\rangle\!\rangle \cup x^{5m}\langle\!\langle x^{3}\rangle\!\rangle  & \mbox{if } m \equiv 1 \Mod 3  \\
			x^{7m}[x^3] \cup x^{8m}[x^3] \cup x^{4m}\langle\!\langle x^{-3}\rangle\!\rangle \cup x^{5m}\langle\!\langle x^{-3}\rangle\!\rangle & \mbox{if } m \equiv 2 \Mod 3.
		\end{array}\right. $$
\item If $t \geq 2$, then $[x]= x^{m}[x^3] \cup x^{2m}[x^3] \cup x^{4m}[x^3] \cup x^{5m}[x^3] \cup x^{7m}[x^3] \cup x^{8m}[x^3]$.
\item If $t \geq 2$, then $$\langle\!\langle x \rangle\!\rangle = \left\{ \begin{array}{ll}
			x^m[x^3] \cup x^{4m}\langle\!\langle x^{-3}\rangle\!\rangle  & \mbox{if } m \equiv 1 \Mod 3  \\
			 x^{2m}[x^3] \cup x^{5m}\langle\!\langle x^{3}\rangle\!\rangle & \mbox{if } m \equiv 2 \Mod 3.
		\end{array}\right. $$
\item If $t \geq 2$, then $$\langle\!\langle x \rangle\!\rangle  = \left\{ \begin{array}{ll}
			x^{7m}[x^3] \cup x^{4m}\langle\!\langle x^{3}\rangle\!\rangle  & \mbox{if } m \equiv 1 \Mod 3  \\
			x^{8m}[x^3] \cup x^{5m}\langle\!\langle x^{-3}\rangle\!\rangle & \mbox{if } m \equiv 2 \Mod 3.
		\end{array}\right.  $$
\item If $t \geq 2$, then $$\langle\!\langle x \rangle\!\rangle = \left\{ \begin{array}{ll}
			x^m[x^3] \cup x^{4m}[x^3]  \cup x^{7m}[x^3]   & \mbox{if } m \equiv 1 \Mod 3  \\
			x^{2m}[x^3] \cup x^{5m}[x^3] \cup x^{8m}[x^3] & \mbox{if } m \equiv 2 \Mod 3.
		\end{array}\right. $$
\end{enumerate}
\end{lema}
\begin{proof}
\begin{enumerate}[label=(\roman*)]
\item Assume that $\ord(x)=3m$ and $m \not\equiv 0 \Mod 3$. Let us take $x^{m+3r} \in x^{m}[x^3]$ for some $r \in G_{m}(1)$.  Then $\gcd(r, m)=1$, and so $\gcd(m+3r, 3m)=1$. Therefore $x^{m}[x^3] \subseteq [x]$. Similarly, we have $x^{2m}[x^3] \subseteq [x]$. Therefore $x^{m}[x^3] \cup x^{2m}[x^3] \subseteq [x]$. Note that $|[x]|= \varphi(3m)=2 \varphi(m)$, $|x^{m}[x^3]| = \varphi(m) = |x^{2m}[x^3]|$, and that $x^{m}[x^3] \cup x^{2m}[x^3]$ is a disjoint union. Thus, the sizes of $[x]$ and $x^m[x^3] \cup x^{2m} [x^3]$ are equal, and therefore $[x]=x^m[x^3] \cup x^{2m} [x^3]$.
\item Assume that $\ord(x)=3m$ and $m \not\equiv 0 \Mod 3$. Let $m \equiv 1 \Mod 3$. We see that $\gcd(r, m)=1$ if and only if $\gcd(m+3r, 3m)=1$. Also $m+3r \equiv 1 \Mod 3$. Therefore
$$x^m[x^3] =\{ x^{m+3r}: r \in G_m(1) \}\subseteq \{ x^{k}: k \in G_{3m,3}^1(1) \} = \langle\!\langle x \rangle\!\rangle. $$ Since the sets $x^m[x^3]$ and $\langle\!\langle x \rangle\!\rangle$ are of equal size, we get $x^m[x^3]= \langle\!\langle x \rangle\!\rangle$. Similarly, if $m \equiv 2 \Mod 3$, we have $x^{2m}[x^3]= \langle\!\langle x \rangle\!\rangle$.
\item Assume that $p=3^{t}m$, $t \geq 2$ and $m \equiv 1 \Mod 3$. Let $x^{m+3r} \in x^{m}[x^3]$ for some $r \in G_{\frac{p}{3}}(1)$.  Then $\gcd(r, \frac{p}{3})=1$, and so $\gcd(m+3r, p)=1$. Thus $x^{m}[x^3] \subseteq [x]$.  Similarly, $x^{2m}[x^3] \subseteq [x]$. Now let $x^{4m+3r} \in x^{4m}\langle\!\langle x^{-3}\rangle\!\rangle$ for some $r \in G_{\frac{p}{3},3}^2(1)$. Again, $\gcd(r, \frac{p}{3})=1$ implies that $\gcd(4m+3r, p)=1$. Therefore $x^{4m}\langle\!\langle x^{-3}\rangle\!\rangle \subseteq [x]$. Similarly, $x^{5m}\langle\!\langle x^{-3}\rangle\!\rangle \subseteq [x]$. Thus $x^m[x^3] \cup x^{2m}[x^3] \cup x^{4m}\langle\!\langle x^{-3}\rangle\!\rangle \cup x^{5m}\langle\!\langle x^{-3}\rangle\!\rangle \subseteq [x]$. Note that $|[x]|=2\times 3^{t-1}\varphi(m)$. Also, $|x^m[x^3]| = 2\times 3^{t-2}\varphi(m) = | x^{2m}[x^3]|$, $|x^{4m}\langle\!\langle x^{-3}\rangle\!\rangle | = 3^{t-2} \varphi(m) = |x^{5m}\langle\!\langle x^{-3}\rangle\!\rangle|$, and that $x^m[x^3] \cup x^{2m}[x^3] \cup x^{4m}\langle\!\langle x^{-3}\rangle\!\rangle \cup x^{5m}\langle\!\langle x^{-3}\rangle\!\rangle$ is a disjoint union. Thus, the sizes of $[x]$ and $x^m[x^3] \cup x^{2m}[x^3] \cup x^{4m}\langle\!\langle x^{-3}\rangle\!\rangle \cup x^{5m}\langle\!\langle x^{-3}\rangle\!\rangle$ are equal, and hence these two sets are equal. For $m \equiv 2 \Mod 3$, the proof follows the similar steps as in the case of $m \equiv 1 \Mod 3$.
\item The proof is similar to the proof Part (iii). For the sake of completeness, we provide the proof  for the case $m \equiv 1 \Mod 3$. Assume that $p=3^{t}m$, $t \geq 2$ and $m \equiv 1 \Mod 3$. Let $x^{7m+3r} \in x^{7m}[x^3]$ for some $r \in G_{\frac{p}{3}}(1)$.  Then $\gcd(r, \frac{p}{3})=1$, and so $\gcd(7m+3r, p)=1$. Thus $x^{7m}[x^3] \subseteq [x]$.  Similarly, $x^{8m}[x^3] \subseteq [x]$. Now let $x^{4m+3r} \in x^{4m}\langle\!\langle x^{3}\rangle\!\rangle$ for some $r \in G_{\frac{p}{3},3}^1(1)$. Again, $\gcd(r, \frac{p}{3})=1$ gives $\gcd(4m+3r, p)=1$. Thus, $x^{4m}\langle\!\langle x^{3}\rangle\!\rangle \subseteq [x]$. Similarly, $x^{5m}\langle\!\langle x^{3}\rangle\!\rangle \subseteq [x]$. Thus $x^{7m}[x^3] \cup x^{8m}[x^3] \cup x^{4m}\langle\!\langle x^{3}\rangle\!\rangle \cup x^{5m}\langle\!\langle x^{3}\rangle\!\rangle \subseteq [x]$. Note that $x^{7m}[x^3] \cup x^{8m}[x^3] \cup x^{4m}\langle\!\langle x^{3}\rangle\!\rangle \cup x^{5m}\langle\!\langle x^{3}\rangle\!\rangle$ is a disjoint union, and so its size is equal to $2 \times 3^{t-2}\varphi(m) + 2 \times 3^{t-2}\varphi(m) + 3^{t-2}\varphi(m) + 3^{t-2}\varphi(m)$, which is equal to the size $2 \times 3^{t-1}\varphi(m)$ of $[x]$. Hence we have the desired equality.

\item Combine Part (iii) and Part (iv), and use $[x^3]=\langle\!\langle x^3 \rangle\!\rangle \cup \langle\!\langle x^{-3} \rangle\!\rangle$ to get the proof of this part.
\item Assume that $p=3^{t}m$, $t \geq 2$ and $m \equiv 1 \Mod 3$. We see that if $r \in G_{\frac{p}{3}}(1)$, then $m+3r \in G_{p,3}^1(1)$. Similarly, if $r \in G_{\frac{p}{3},3}^2(1)$, then $4m+3r \in G_{p,3}^1(1)$. Thus we have $x^m[x^3] \cup x^{4m}\langle\!\langle x^{-3}\rangle\!\rangle \subseteq \langle\!\langle x\rangle\!\rangle$. Since the sizes of $x^m[x^3] \cup x^{4m}\langle\!\langle x^{-3}\rangle\!\rangle$ and $\langle\!\langle x\rangle\!\rangle$ are equal, we find that $x^m[x^3] \cup x^{4m}\langle\!\langle x^{-3}\rangle\!\rangle = \langle\!\langle x\rangle\!\rangle$. Similarly,  we have $x^{2m}[x^3] \cup x^{5m}\langle\!\langle x^{3}\rangle\!\rangle = \langle\!\langle x\rangle\!\rangle$ for $m \equiv 2 \Mod 3$.
\item The proof of this part follows similar steps as in Part (vi). For the sake of completeness, we provide the proof for the case $m \equiv 2 \Mod 3$. Assume that $p=3^{t}m$, $t \geq 2$ and $m \equiv 2 \Mod 3$.  We see that if $r \in G_{\frac{p}{3}}(1)$, then $8m+3r \in G_{p,3}^1(1)$. Also, if $r \in G_{\frac{p}{3},3}^2(1)$, then $5m+3r \in G_{p,3}^1(1)$. Thus $x^{8m}[x^3] \cup x^{5m}\langle\!\langle x^{-3}\rangle\!\rangle   \subseteq  \langle\!\langle x \rangle\!\rangle$. Since the sizes of $x^{8m}[x^3] \cup x^{5m}\langle\!\langle x^{-3}\rangle\!\rangle$ and $\langle\!\langle x \rangle\!\rangle$ are equal, we find that $x^{8m}[x^3] \cup x^{5m}\langle\!\langle x^{-3}\rangle\!\rangle   =  \langle\!\langle x \rangle\!\rangle$.
\item Combine Part (vi) and Part (vii), and use $[x^3]=\langle\!\langle x^3 \rangle\!\rangle \cup \langle\!\langle x^{-3} \rangle\!\rangle$ to get the proof of this part. \qedhere
\end{enumerate}
\end{proof}

For $ x\in \Gamma$, define $S_x^1:=\text{\footnotesize$\bigcup\limits_{s\in {\rm Cl}(x)}$} [ s ]$. We see that if $m=\ord(x)$, then 
\[S_x^1=\{g^{-1}x^r g\colon g\in \Gamma, r\in G_m(1)\}= \text{\footnotesize$\bigcup\limits_{s\in [x]}$}{\rm Cl}(s).\]
The set $S_x^1$ is also known as the rational conjugacy class\index{rational conjugacy class} of $x$. See~\cite{foster2016spectra} for details. For each $y\in S_x^1$, it is clear that ${\rm Cl}(y), [y]\subseteq S_x^1$. Now let $A$ be a symmetric subset of $\Gamma$ such that $x\in A$, and ${\rm Cl}(a), [a]\subseteq A$ for each $a\in A$. Let $g^{-1}x^r g\in S_x^1$, where $g\in \Gamma$, $r\in G_m(1)$ and $m=\ord(x)$. As $[x]\subseteq A$, we have $x^r\in A$. Now ${\rm Cl}(x^r)\subseteq A$, and so $g^{-1}x^r g\in A$. Thus $S_x^1 \subseteq A$, and therefore $S_x^1$ is the smallest symmetric subset of $\Gamma$ containing $x$ that is closed under both conjugacy and the equivalence relation $\sim$. Considering each of the  repeated equivalence classes, if any, only once  in $\text{\footnotesize$\bigcup\limits_{s\in {\rm Cl}(x)}$} [ s ]$, we can write $S_x^1=\text{\footnotesize$\bigcup\limits_{i=1}^{\ell}$}[x_i]$, where the equivalence classes $[x_1], \ldots, [x_{\ell}]$ are distinct. We state this fact in the next lemma.

\begin{lema}\label{SymbolSetChara1}
If $x \in \Gamma$, then there exist distinct equivalence classes $[x_1], \ldots, [x_{\ell}]$  such that $S_x^1= \text{\footnotesize$\bigcup\limits_{i=1}^{\ell}$}[x_i]$, where $x_1,\ldots,x_{\ell}\in {\rm Cl}(x)$.
\end{lema}

\begin{lema}\label{partition1}If $y\in S_x^1$, then $S_y^1=S_x^1$.
\end{lema}
\begin{proof}Let $y\in S_x^1$, so that $y=g^{-1}x^r g$ for some $g\in \Gamma$ and $r\in G_m(1)$, where $m=\ord(x)$. We see that $\ord(y)=\ord(x)=m$. Now let $z\in S_y^1$. Then $z=h^{-1}y^t h$ for some $h\in \Gamma$ and $t\in G_m(1)$. This gives $z=h^{-1}y^t h = h^{-1}g^{-1}x^{rt} g h\in S_x^1$. Conversely, let $w\in S_x^1$ so that $w=h^{-1}x^t h$ for some $h\in \Gamma$ and $t\in G_m(1)$. Therefore \[w=h^{-1}x^t h = (h^{-1}g)g^{-1}(x^r)^{r^{-1}t}g( g^{-1} h) =(h^{-1}g)y^{r^{-1}t}  ( g^{-1} h)\in S_y^1.\] Here $r^{-1}$ is the multiplicative inverse of $r$ in the group $G_m(1)$. Hence we conclude that $S_y^1=S_x^1$.
\end{proof}
Due to Lemma~\ref{partition1}, the sets $S_x^1$ and $S_y^1$ are either disjoint or equal. Hence the class of distinct subsets of $\Gamma$ of the form $S_x^1$ is a partition of $\Gamma$.

Let $x\in \Gamma(3)$ be an element of order $m$. The element $x$ is said to be \textit{tolerable}\index{tolerable} if $x^r \not\in {\rm Cl}(x)$ for all $r \in G_{m,3}^2(1)$. 
The following lemma characterizes tolerable elements in terms of skew-symmetric sets. 
\begin{lema}\label{iff2}
If $x\in \Gamma(3)$, then $x$ is tolerable if and only if the set $\text{\footnotesize$\bigcup\limits_{s\in {\rm Cl}(x)}$} \langle\!\langle s \rangle\!\rangle$ is  skew-symmetric.
\end{lema} 
\begin{proof}We see that if $m=\ord(x)$, then 
\[ \text{\footnotesize$\bigcup\limits_{s\in {\rm Cl}(x)}$} \langle\!\langle s \rangle\!\rangle =\{g^{-1}x^r g\colon g\in \Gamma, r\in G_{m,3}^1(1)\}= \text{\footnotesize$\bigcup\limits_{s\in \langle\!\langle x\rangle\!\rangle}$}{\rm Cl}(s).\]
Assume that $x$ is not tolerable, so that $x^r\in {\rm Cl}(x)$ for some $r\in G_{m,3}^2(1)$. As $m-r\in G_{m,3}^1(1)$ and ${\rm Cl}(x)\subseteq \text{\footnotesize$\bigcup\limits_{s\in {\rm Cl}(x)}$} \langle\!\langle s \rangle\!\rangle$, we find that $x^r,x^{m-r}\in \text{\footnotesize$\bigcup\limits_{s\in {\rm Cl}(x)}$} \langle\!\langle s \rangle\!\rangle$. Hence $\text{\footnotesize$\bigcup\limits_{s\in {\rm Cl}(x)}$} \langle\!\langle s \rangle\!\rangle$ is not skew-symmetric. 

On the other hand, assume that $\text{\footnotesize$\bigcup\limits_{s\in {\rm Cl}(x)}$} \langle\!\langle s \rangle\!\rangle$ is not a skew-symmetric set. Then there is an\linebreak[4] $y=g^{-1}x^r g\in \text{\footnotesize$\bigcup\limits_{s\in {\rm Cl}(x)}$} \langle\!\langle s \rangle\!\rangle$ for some $r \in G_{m,3}^1(1)$ such that $y^{-1}\in \text{\footnotesize$\bigcup\limits_{s\in {\rm Cl}(x)}$} \langle\!\langle s \rangle\!\rangle$. Therefore we have $g^{-1}x^{m-r} g=y^{-1}=h^{-1}x^k h$ for some $h\in \Gamma, k\in G_{m,3}^1(1)$. Let $t\in G_m(1)$ be the multlipicative inverse of $m-r$. We have $g^{-1}x^{(m-r)t} g=h^{-1}x^{kt} h$, and it gives $x^{kt}=hg^{-1}x gh^{-1}\in {\rm Cl}(x)$. Since $(m-r)t\equiv 1 \Mod 3$ and $m-r\in G_{m,3}^2(1)$, we have that $t\in G_{m,3}^2(1)$.  Thus $kt\in G_{m,3}^2(1)$ with $x^{kt}\in {\rm Cl}(x)$, giving that $x$ is not tolerable. 
\end{proof}

Let $ x\in \Gamma(3)$ be tolerable, and define $S_x^3:=\text{\footnotesize$\bigcup\limits_{s\in {\rm Cl}(x)}$} \langle\!\langle s \rangle\!\rangle$. The structure and properties of the set $S_x^3$ are similar to those of $S_x^1$ and $S_x^4$.  If $\Gamma$ is abelian, then $S_x^3=\langle\!\langle x \rangle\!\rangle$ for each $x \in \Gamma(3)$. For each $y\in S_x^3$, it is clear that ${\rm Cl}(y), \langle\!\langle y \rangle\!\rangle\subseteq S_x^3$. Now let $A$ be a skew-symmetric subset of $\Gamma$ containing a tolerable element $x$, and ${\rm Cl}(a), \langle\!\langle a\rangle\!\rangle\subseteq A$ for each $a\in A$. It is easy to see that $S_x^3 \subseteq A$. Thus, $S_x^3$ is the smallest skew-symmetric subset of $\Gamma$ containing $x$ that is closed under both conjugacy and the equivalence relation $\simeq$. Considering each of the  repeated equivalence classes, if any, only once in $\text{\footnotesize$\bigcup\limits_{s\in {\rm Cl}(x)}$} \langle\!\langle s \rangle\!\rangle$, we can write $S_x^3=\text{\footnotesize$\bigcup\limits_{i=1}^r$}\langle\!\langle y_i \rangle\!\rangle$, where the equivalence classes $\langle\!\langle y_1 \rangle\!\rangle, \ldots, \langle\!\langle y_r \rangle\!\rangle$ are distinct. We state this fact in the next lemma.

\begin{lema}\label{SymbolSetChara2ch5}
If $x$ is a tolerable element in $\Gamma(3)$, then there are distinct equivalence classes $\langle\!\langle x_1 \rangle\!\rangle, \ldots, \langle\!\langle x_r \rangle\!\rangle$ such that $S_x^3= \text{\footnotesize$\bigcup\limits_{i=1}^r$}\langle\!\langle x_i \rangle\!\rangle$, where $x_1,\ldots,x_r\in {\rm Cl}(x)$.
\end{lema} 

\begin{lema}\label{partition3}If $y\in S_x^3$, then $S_y^3=S_x^3$.
\end{lema}
\begin{proof}Let $y\in S_x^3$, so that $y=g^{-1}x^r g$ for some $g\in \Gamma$ and $r\in G_{m,3}^1(1)$, where $m=\ord(x)$. We see that $\ord(y)=\ord(x)=m$. Now let $z\in S_y^3$. Then $z=h^{-1}y^t h$ for some $h\in \Gamma$ and $t\in G_{m,3}^1(1)$. This gives $z=h^{-1}y^t h = h^{-1}g^{-1}x^{rt} g h\in S_x^3$. Conversely, let $w\in S_x^3$ so that $w=h^{-1}x^t h$ for some $h\in \Gamma$ and $t\in G_{m,3}^1(1)$. Therefore \[w=h^{-1}x^t h = (h^{-1}g)g^{-1}(x^r)^{r^{-1}t}g( g^{-1} h) =(h^{-1}g)y^{r^{-1}t}  ( g^{-1} h)\in S_y^3.\] Here $r^{-1}$ is the multiplicative inverse of $r$ in the subgroup $G_{m,3}^1(1)$. Thus we conclude that $S_y^3=S_x^3$.
\end{proof}
Due to Lemma~\ref{partition3}, the sets $S_x^3$ and $S_y^3$ are either disjoint or equal. 

\begin{lema}\label{NewLemmaEquivaTrans2Ch5} Let $ x\in \Gamma(3)$. If $S_x^1 = [ x_1 ] \cup \cdots \cup [ x_k ]$ for some $x_1,\ldots,x_k\in \Cl(x)$, then $S_{x^3}^1= [ x_1^{3} ] \cup \cdots \cup [ x_k^{3} ]$.
\end{lema}
\begin{proof}
Let $m=\ord(x)$ and $S_x^1 = [ x_1 ] \cup \cdots \cup [ x_k ]$ for some $x_1\ldots,x_k\in {\rm Cl}(x)$. Assume that the sets $[ x_1 ], \ldots , [ x_k ]$ are all distinct. We see that
\begin{equation*}
\begin{split}
S_{x^3}^1=&\left\{g^{-1}x^{3r} g\colon g\in \Gamma, r\in G_{\frac{m}{3}}(1)\right\}\\
=&\left\{g^{-1}x^{3r} g\colon g\in \Gamma, r\in G_{\frac{m}{3}}(1)\right\} \cup \left\{g^{-1}x^{3(\frac{m}{3}+r)} g\colon g\in \Gamma, r\in G_{\frac{m}{3}}(1)\right\}\\
& \cup \left\{g^{-1}x^{3(\frac{2m}{3}+r)} g\colon g\in \Gamma, r\in G_{\frac{m}{3}}(1)\right\}\\
=&\left\{g^{-1}x^{3r} g\colon g\in \Gamma, r\in G_{m}(1), r <\frac{m}{3} \right\} \cup \left\{g^{-1}x^{3t} g\colon g\in \Gamma, t\in G_{m}(1),\frac{m}{3} < t < \frac{2m}{3}\right\}\\
&\cup \left\{g^{-1}x^{3t} g\colon g\in \Gamma, t\in G_{m}(1),\frac{2m}{3} < t \right\}\\
=&\left\{g^{-1}x^{3r} g\colon g\in \Gamma, r\in G_{m}(1) \right\}\\
=&\left\{y^3\colon y\in S_x^1\right\}.
\end{split}
\end{equation*} 
Now noting that $\{s^3\colon s\in [x]\} =[x^3]$ and $S_x^1 =[ x_1 ] \cup \cdots \cup [ x_k ]$, we have $S_{x^3}^1= [ x_1^3 ] \cup \cdots \cup [ x_k^3 ]$.
\end{proof}

\begin{lema} If $x\in \Gamma(3)$ is tolerable, then $S_x^3 \cup S_{x^{-1}}^3=S_x^1$.
\end{lema}
\begin{proof}Let $m=\ord(x)$. We have
\begin{align*}
S_x^3 \cup S_{x^{-1}}^3= & \left\{g^{-1}x^{r} g\colon g\in \Gamma, r\in G_{m,3}^1(1)\right\} \cup \left\{g^{-1}x^{-r} g\colon g\in \Gamma, r\in G_{m,3}^1(1)\right\}\\
=& \left\{g^{-1}x^{r} g\colon g\in \Gamma, r\in G_{m,3}^1(1)\right\} \cup \left\{g^{-1}x^{r} g\colon g\in \Gamma, r\in G_{m,3}^2(1)\right\}\\
=& \left\{g^{-1}x^{r} g\colon g\in \Gamma, r\in G_m(1)\right\}\\
=& S_x^1. \qedhere
\end{align*}
\end{proof}

\begin{lema}\label{NewLemmaEquivaTrans3Ch5} Let $x\in \Gamma(3)$ be a tolerable element. If $S_x^3 = \langle\!\langle  x_1 \rangle\!\rangle  \cup \cdots \cup \langle\!\langle  x_k \rangle\!\rangle$ for some $x_1,\ldots,x_k\in \Cl(x)$, then $S_{x^3}^1= [x_1^3] \cup \cdots \cup [x_k^3]$.
\end{lema}
\begin{proof} 
Assume that $S_x^3 = \langle\!\langle  x_1 \rangle\!\rangle  \cup \cdots \cup \langle\!\langle  x_k \rangle\!\rangle$ for some $x_1,\ldots,x_k\in {\rm Cl}(x)$. Then we have $S_{x^{-1}}^3 = \langle\!\langle  x_1^{-1} \rangle\!\rangle  \cup \cdots \cup \langle\!\langle  x_k^{-1} \rangle\!\rangle$. Therefore
\begin{align*}
S_x^1 =S_x^3 \cup S_{x^{-1}}^3 = (\langle\!\langle  x_1 \rangle\!\rangle \cup \langle\!\langle  x_1^{-1} \rangle\!\rangle )  \cup \cdots \cup (\langle\!\langle  x_1 \rangle\!\rangle \cup \langle\!\langle  x_k^{-1} \rangle\!\rangle)
=[ x_1 ]  \cup \cdots \cup [ x_k ].
\end{align*}
Now the result follows from Lemma~\ref{NewLemmaEquivaTrans2Ch5}.\qedhere
\end{proof}

\noindent For $x \in \Gamma$ and $j\in \{1,\ldots,h\}$, define 
$$C_x(j):=\frac{1}{\chi_j({\mathbf 1})}\sum_{s \in S^1_x } \chi_{j}(s).$$ 
Note that $S_x^1 \in \mathbb{B}(\Gamma)$ and $C_x(j)$ is an eigenvalue of the normal undirected Cayley graph ${\rm Cay}(\Gamma, S^1_x)$. As a consequence of Theorem~\ref{NorMixCayGraphInteg}, $C_x(j)$ is an integer for each $x \in \Gamma$ and $j\in \{1,\ldots,h\}$. 

\begin{lema}\label{NewSum4mLemma11Ch5} Let $ x\in \Gamma$ and $\ord(x)=3^{t}m$. If  $m\not\equiv 0 \Mod 3$ and $t \geq 2$, then
$$2C_x(j)=\bigg(\sum\limits_{s\in G_{9}(1)} \chi_j(x^{sm})\bigg) C_{x^3}(j).$$ 
Moreover, $\frac{C_x(j)}{3}$ is an integer for each $j \in \{ 1,\ldots , h \}$.
\end{lema}
\begin{proof} Let $S_x^1= [ x_1 ]  \cup \cdots \cup [ x_k ]$ for some $x_1,\ldots ,x_k\in \Cl(x)$ and $j \in \{ 1,\ldots , h \}$. We use the fact that each $[x_i]$ can be written as disjoint unions in two different ways using Part (iii) and Part (iv) of Lemma~\ref{NewLemmaEquivaTrans1Ch5}. 
For $m \equiv 1 \Mod 3$, using Part (iii) and Part (iv) of Lemma~\ref{NewLemmaEquivaTrans1Ch5}, we have 
\begin{align}
2 \sum_{s \in [ x_i ] } \chi_j(s) =& \sum_{s \in [ x_i ] } \chi_j(s) + \sum_{s \in [ x_i ] } \chi_j(s) \nonumber\\
 =& \sum_{s \in  x_i^m[x_i^3]} \chi_j(s) + \sum_{s \in  x_i^{2m}[x_i^3]} \chi_j(s) + \sum_{s \in  x_i^{4m}\langle\!\langle x_i^{-3}\rangle\!\rangle} \chi_j(s) + \sum_{s \in  x_i^{5m}\langle\!\langle x_i^{-3}\rangle\!\rangle} \chi_j(s) \nonumber\\
&+ \sum_{s \in  x_i^{7m}[x_i^3]} \chi_j(s) + \sum_{s \in x_i^{8m}[x_i^3] } \chi_j(s) + \sum_{s \in  x_i^{4m}\langle\!\langle x_i^{3}\rangle\!\rangle} \chi_j(s) + \sum_{s \in  x_i^{5m}\langle\!\langle x_i^{3}\rangle\!\rangle} \chi_j(s) \nonumber\\
=& \sum_{s \in [x_i^3]} \chi_j(x_i^{m}) \chi_j(s) + \sum_{s \in  [x_i^3]} \chi_j(x_i^{2m}) \chi_j(s) + \sum_{s \in  [x_i^3]} \chi_j(x_i^{4m})\chi_j(s) \nonumber\\ 
&+ \sum_{s \in [x_i^3]} \chi_j(x_i^{5m}) \chi_j(s) + \sum_{s \in [x_i^3]} \chi_j(x_i^{7m})\chi_j(s) + \sum_{s \in [x_i^3] } \chi_j(x_i^{8m})\chi_j(s) \label{NewEqFinalCorrCh53}
\end{align}
for each $i \in \{ 1, \ldots, k\}$. Similarly, for $m \equiv 2 \Mod 3$, using Part (iii) and Part (iv) of Lemma~\ref{NewLemmaEquivaTrans1Ch5}, we have 
\begin{align}
2 \sum_{s \in [ x_i ] } \chi_j(s) =& \sum_{s \in [x_i^3]} \chi_j(x_i^{m}) \chi_j(s) + \sum_{s \in  [x_i^3]} \chi_j(x_i^{2m}) \chi_j(s) + \sum_{s \in  [x_i^3]} \chi_j(x_i^{4m})\chi_j(s) \nonumber\\ 
&+ \sum_{s \in [x_i^3]} \chi_j(x_i^{5m}) \chi_j(s) + \sum_{s \in [x_i^3]} \chi_j(x_i^{7m})\chi_j(s) + \sum_{s \in [x_i^3] } \chi_j(x_i^{8m})\chi_j(s) \label{NewEqFinalCorrCh54}
\end{align} for each $i \in \{ 1, \ldots, k\}$. Thus using Equations~(\ref{NewEqFinalCorrCh53}) and (\ref{NewEqFinalCorrCh54}), we get
\begin{align} \label{EqNewSum4mLemma11Ch5}
2C_x(j) =& \frac{1}{\chi_j(\mathbf 1)} \sum_{i=1}^{k} 2 \sum_{s \in [ x_i ] } \chi_j(s) \nonumber\\
=& \frac{1}{\chi_j(\mathbf 1)} \sum_{i=1}^{k} \bigg( \sum_{s \in [x_i^3] } \chi_j(x_i^m) \chi_j(s) + \sum_{s \in [ x_i^3 ] } \chi_j(x_i^{2m}) \chi_j(s) + \sum_{s \in [x_i^3] } \chi_j(x_i^{4m}) \chi_j(s)\nonumber\\
& + \sum_{s \in [ x_i^3 ] } \chi_j(x_i^{5m}) \chi_j(s) + \sum_{s \in [x_i^3] } \chi_j(x_i^{7m}) \chi_j(s) + \sum_{s \in [ x_i^3 ] } \chi_j(x_i^{8m}) \chi_j(s) \bigg) \nonumber\\
=&  \bigg( \chi_j(x^{m}) + \chi_j(x^{2m}) + \chi_j(x^{4m}) + \chi_j(x^{5m}) + \chi_j(x^{7m}) \nonumber\\
& + \chi_j(x^{8m})\bigg) \frac{1}{\chi_j(\mathbf 1)} \sum_{i=1}^{k}  \sum_{s \in [ x_i^3 ] }  \chi_j(s) \nonumber\\
=&  \bigg( \sum\limits_{r\in G_{9}(1)} \chi_j(x^{rm})\bigg) C_{x^3}(j).
\end{align}
Here the third equality in Equation (\ref{EqNewSum4mLemma11Ch5}) follows from the fact that $x_1,\ldots ,x_k\in \Cl(x)$, and  the fourth equality in Equation (\ref{EqNewSum4mLemma11Ch5}) follows from Lemma~\ref{NewLemmaEquivaTrans2Ch5}. 

Let $d_j=\chi_j(\mathbf 1)$. We apply induction on $t$ to prove that $\frac{C_x(j)}{3}$ is an integer. Let $t=2$, so that $\ord(x)=9m$ with $m \not\equiv 0 \Mod 3$. By Theorem~\ref{NewThmChap1Added}, we have $\chi_j(x^{m}) = \sum\limits_{\ell=1}^{d_j} \epsilon_{j\ell}$, where $\epsilon_{j1},\ldots ,\epsilon_{jd_j}$ are some $9$-th roots of unity. We have
\begin{equation} \label{EqNewSum4mLemma55}
\begin{split}
\sum\limits_{r\in G_{9}(1)} \chi_j(x^{rm}) &=\sum\limits_{r\in G_{9}(1)} \sum\limits_{\ell=1}^{d_j} \epsilon_{j\ell}^r
= \sum\limits_{\ell=1}^{d_j} \sum\limits_{r\in G_{9}(1)} \epsilon_{j\ell}^r.
\end{split} 
\end{equation} 
Note that $\sum\limits_{r\in G_{9}(1)} \epsilon_{j\ell}^r = (\epsilon_{j\ell} + \epsilon_{j\ell}^2)(1+\epsilon_{j\ell}^3 + \epsilon_{j\ell}^6)$. Since $\epsilon_{j\ell} \in \{1,\omega_9,\omega_9^2,\ldots , \omega_9^8\}$, we have 
$$ \sum\limits_{r\in G_{9}(1)} \epsilon_{j\ell}^r=    \left\{ \begin{array}{rl}
					6 & \mbox{if } \epsilon_{j\ell}=1 \\
					-3  & \mbox{if } \epsilon_{j\ell} \in \{\omega_9^3,\omega_9^6\}\\
					0  & \mbox{otherwise.}
				\end{array}\right.$$
Thus, $\sum\limits_{r\in G_{9}(1)}\epsilon_{j\ell}^r$ is an integer multiple of $3$ for each $ \ell \in \{ 1, \ldots ,d_j\}$. Therefore by Equation~(\ref{EqNewSum4mLemma55}), $\sum\limits_{r\in G_{9}(1)} \chi_j(x^{rm})$ is an integer multiple of $3$. Now Equation~(\ref{EqNewSum4mLemma11Ch5}) gives that $\frac{2C_x(j)}{3}$ is an integer. Since $C_x(j)$ is an integer, integrality of $\frac{2C_x(j)}{3}$ gives that $\frac{C_x(j)}{3}$ is also an integer. 

Assume that $\frac{C_y(j)}{3}$ is an integer for each $j \in \{ 1,\ldots ,h\}$ whenever $\ord(y)=3^{t-1}m$ with $m \not\equiv 0 \Mod 3$ and $t\geq 3$. Let $\ord(x)=3^tm$ with $m \not\equiv 0 \Mod 3$ and $t\geq 3$. Note that $\ord(x^3)=3^{t-1}m$. Therefore by induction hypothesis, $\frac{C_{x^3}(j)}{3}$ is an  integer. By Equation~(\ref{EqNewSum4mLemma11Ch5}), $\sum\limits_{s\in G_{9}(1)} \chi_j(x^{sm})$ is a rational algebraic integer whenever $C_{x^3}(j)\neq 0$. Thus, if $C_{x^3}(j)\neq 0$ then $\sum\limits_{s\in G_{9}(1)} \chi_j(x^{sm})$ is an integer. Therefore by Equation~(\ref{EqNewSum4mLemma11Ch5}), $\frac{2C_x(j)}{3}$ is an integer, and accordingly $\frac{C_x(j)}{3}$ is an integer. Hence the proof is complete by induction.
\end{proof}

Let $x \in \Gamma(3)$ be tolerable. For each $j \in \{ 1,\ldots , h \}$, define $$T_x(j):= \frac{1}{\chi_j(\mathbf 1)} \sum_{s \in S^3_x} \i \sqrt{3} ( \chi_{j}(s)-\chi_{j}(s^{-1})).$$ 

Let $j \in \{ 1,\ldots ,h\}$. Using $S^1_x=S^3_x \cup S^3_{x^{-1}}$, we see that
\begin{align*}
\frac{C_x(j)+T_x(j)}{2} &= \frac{1}{2\chi_j(\mathbf 1)} \bigg[ \sum_{s \in S^1_x}  \chi_{j}(s) +  \sum_{s \in S^3_x} \i \sqrt{3} ( \chi_{j}(s)-\chi_{j}(s^{-1})) \bigg]\nonumber\\
&= \frac{1}{2\chi_j(\mathbf 1)} \bigg[ \sum_{s \in S^3_x}  \chi_{j}(s) + \sum_{s \in S^3_{x^{-1}}}  \chi_{j}(s) +  \sum_{s \in S^3_x} \i \sqrt{3} ( \chi_{j}(s)-\chi_{j}(s^{-1})) \bigg]\nonumber\\
&= \frac{1}{\chi_j(\mathbf 1)} \bigg[ \sum_{s \in S^3_x} \left( \omega_6 \chi_{j}(s) +  \omega_6^5 \chi_{j}(s^{-1}) \right) \bigg].
\end{align*}
Thus $\frac{C_x(j)+T_x(j)}{2}$ is an HS-eigenvalue of the normal oriented Cayley graph $\Cay(\Gamma, S^3_x)$. Therefore by Theorem~\ref{NorOriCayGraphIntegCh5}, $\frac{C_x(j)+T_x(j)}{2}$ is an integer. Since $C_x(j)$ is an integer (by Theorem~\ref{NorMixCayGraphInteg}), $T_x(j)$ is also an integer for each $j \in \{ 1,\ldots ,h\}$.

\begin{lema}\label{NewSum4mLemma22} Let $x\in \Gamma(3)$ be tolerable and $\ord(x)=3m$. If $m\not\equiv 0 \Mod 3$, then
$$T_x(j)= \left\{ \begin{array}{ll}
					-2\sqrt{3} \Im(\chi_j(x^m)) C_{x^3}(j) & \mbox{if } m \equiv 1 \Mod 3 \\
					-2\sqrt{3} \Im(\chi_j(x^{2m}))C_{x^3}(j)  & \mbox{if } m \equiv 2 \Mod 3.
				\end{array}\right.$$ Moreover, $\frac{T_x(j)}{3}$ is an integer for each $j \in \{ 1,\ldots , h \}$.
\end{lema}
\begin{proof} Let $S_x^3= \langle\!\langle x_1 \rangle\!\rangle \cup \cdots \cup \langle\!\langle x_k \rangle\!\rangle$ for some $x_1,\ldots,x_k\in \Cl(x)$ and $j \in \{ 1,\ldots , h \}$. We get
\begin{equation*}
\begin{split}
T_x(j) &= \frac{1}{\chi_j(\mathbf 1)} \sum_{i=1}^{k}  \sum_{s \in \langle\!\langle x_i \rangle\!\rangle}\i \sqrt{3}( \chi_j(s)-\chi_j(s^{-1}))\\
&= \left\{ \begin{array}{ll}
					 \frac{1}{\chi_j(\mathbf 1)}\sum\limits_{i=1}^{k}  \sum\limits_{s \in [ x_i^3 ]}\i \sqrt{3} \big[ \chi_j(x_i^m) \chi_j(s)-\chi_j(x_i^{-m})\chi_j(s^{-1})\big] & \mbox{if } m \equiv 1 \Mod 3 \\
					 \frac{1}{\chi_j(\mathbf 1)}\sum\limits_{i=1}^{k}  \sum\limits_{s \in [ x_i^{3} ]} \i \sqrt{3}\big[ \chi_j(x_i^{2m})\chi_j(s)-\chi_j(x_i^{-2m})\chi_j(s^{-1})\big] & \mbox{if } m \equiv 2 \Mod 3
				\end{array}\right.\\
				&= \left\{ \begin{array}{ll}
					 \frac{\i \sqrt{3}}{\chi_j(\mathbf 1)}\sum\limits_{i=1}^{k}  \bigg[ \chi_j(x_i^m)\sum\limits_{s \in [ x_i^3 ]} \chi_j(s)-\overline{\chi_j(x_i^{m})} \sum\limits_{s \in [ x_i^3 ]}\chi_j(s^{-1})\bigg] & \mbox{if } m \equiv 1 \Mod 3 \\
					 \frac{\i \sqrt{3}}{\chi_j(\mathbf 1)}\sum\limits_{i=1}^{k}  \bigg[ \chi_j(x_i^{2m}) \sum\limits_{s \in [ x_i^{3} ]} \chi_j(s)-\overline{\chi_j(x_i^{2m})} \sum\limits_{s \in [ x_i^{3} ]} \chi_j(s^{-1})\bigg] & \mbox{if } m \equiv 2 \Mod 3
				\end{array}\right.\\
					&= \left\{ \begin{array}{ll}
					 \frac{\i \sqrt{3}}{\chi_j(\mathbf 1)}\sum\limits_{i=1}^{k}  \bigg[ \chi_j(x_i^m)\sum\limits_{s \in [ x_i^3 ]} \chi_j(s)-\overline{\chi_j(x_i^{m})} \sum\limits_{s \in [ x_i^3 ]}\chi_j(s)\bigg] & \mbox{if } m \equiv 1 \Mod 3 \\
					 \frac{\i \sqrt{3}}{\chi_j(\mathbf 1)}\sum\limits_{i=1}^{k}  \bigg[ \chi_j(x_i^{2m}) \sum\limits_{s \in [ x_i^{3} ]} \chi_j(s)-\overline{\chi_j(x_i^{2m})} \sum\limits_{s \in [ x_i^{3} ]} \chi_j(s)\bigg] & \mbox{if } m \equiv 2 \Mod 3
				\end{array}\right.\\
				&= \left\{ \begin{array}{ll}
					-2\sqrt{3} \Im(\chi_j(x^m)) \frac{1}{\chi_j(\mathbf 1)} \sum\limits_{i=1}^{k}  \sum\limits_{s \in [ x_i^{3} ]} \chi_j(s) & \mbox{if } m \equiv 1 \Mod 3 \\
					-2\sqrt{3} \Im(\chi_j(x^{2m})) \frac{1}{\chi_j(\mathbf 1)}  \sum\limits_{i=1}^{k}  \sum\limits_{s \in [ x_i^{3} ]} \chi_j(s)  & \mbox{if } m \equiv 2 \Mod 3
				\end{array}\right.\\
&= \left\{ \begin{array}{ll}
					-2\sqrt{3} \Im(\chi_j(x^m)) C_{x^3}(j) & \mbox{if } m \equiv 1 \Mod 3 \\
					-2\sqrt{3} \Im(\chi_j(x^{2m}))C_{x^3}(j)  & \mbox{if } m \equiv 2 \Mod 3.
				\end{array}\right.\\
\end{split} 
\end{equation*} 
Here the second equality follows from Part (ii) of Lemma~\ref{NewLemmaEquivaTrans1Ch5}, and the fourth equality follows from Lemma~\ref{NewLemmaEquivaTrans3Ch5}. Let $d_j=\chi_j(\mathbf 1)$. By Theorem~\ref{NewThmChap1Added}, we have $\chi_j(x^{m}) = \sum\limits_{\ell=1}^{d_j} \epsilon_{j\ell}$, where $\epsilon_{j1},\ldots ,\epsilon_{jd_j}$ are cube roots of unity. Therefore, $2\sqrt{3}\Im(\chi_j(x^m))$ is an integer multiple of $3$. Similarly, $2\sqrt{3}\Im(\chi_j(x^{2m}))$ is also an integer multiple of $3$. Hence $\frac{T_x(j)}{3}$ is an integer for each $j \in \{ 1,\ldots , h \}$.
\end{proof}

\begin{lema}\label{NewSum4mLemma33Ch5} Let $ x\in \Gamma$ be tolerable and $\ord(x)=3^{t}m$. If $m\not\equiv 0 \Mod 3$ and $t \geq 2$, then
$$2T_x(j)= \left\{ \begin{array}{ll}
					-2\sqrt{3}\bigg(\sum\limits_{s\in G_{9,3}^1(1)} \Im(\chi_j(x^{sm})) \bigg) C_{x^3}(j) & \mbox{if } m \equiv 1 \Mod 3 \\
					-2\sqrt{3}\bigg(\sum\limits_{s\in G_{9,3}^2(1)} \Im(\chi_j(x^{sm})) \bigg) C_{x^3}(j)  & \mbox{if } m \equiv 2 \Mod 3.
				\end{array}\right.$$ Moreover, $\frac{T_x(j)}{3}$ is an integer for each $j \in \{ 1,\ldots , h \}$.
\end{lema}
\begin{proof} Let $S_x^3= \langle\!\langle x_1 \rangle\!\rangle \cup \cdots \cup \langle\!\langle x_k \rangle\!\rangle$ for some $x_1,\ldots,x_k\in \Cl(x)$ and $j \in \{ 1,\ldots , h \}$.  We use the fact that each $\langle\!\langle x_i \rangle\!\rangle$ can be written as disjoint unions in two different ways using Part (vi) and Part (vii) of Lemma~\ref{NewLemmaEquivaTrans1Ch5}. For $m \equiv 1 \Mod 3$, using Part (vi) and Part (vii) of Lemma~\ref{NewLemmaEquivaTrans1Ch5}, we have
\begin{align}
&~~2\sum_{s \in \langle\!\langle x_i \rangle\!\rangle} \i \sqrt{3}( \chi_j(s)-\chi_j(s^{-1})) \nonumber\\
=& \sum_{s \in \langle\!\langle x_i \rangle\!\rangle} \i \sqrt{3}( \chi_j(s)-\chi_j(s^{-1})) + \sum_{s \in \langle\!\langle x_i \rangle\!\rangle} \i \sqrt{3}( \chi_j(s)-\chi_j(s^{-1})) \nonumber\\
=&\sum_{s \in x_i^m[x_i^3] } \i \sqrt{3}( \chi_j(s)-\chi_j(s^{-1})) + \sum_{s \in x_i^{4m}\langle\!\langle x_i^{-3}\rangle\!\rangle } \i \sqrt{3}( \chi_j(s)-\chi_j(s^{-1})) \nonumber\\
& + \sum_{s \in x_i^{7m}[x_i^3] } \i \sqrt{3}( \chi_j(s)-\chi_j(s^{-1})) + \sum_{s \in x_i^{4m}\langle\!\langle x_i^{3}\rangle\!\rangle } \i \sqrt{3}( \chi_j(s)-\chi_j(s^{-1})) \nonumber\\
=&\sum_{s \in x_i^m[x_i^3] } \i \sqrt{3}( \chi_j(s)-\chi_j(s^{-1})) + \sum_{s \in x_i^{4m}[ x_i^{3} ] } \i \sqrt{3}( \chi_j(s)-\chi_j(s^{-1})) \nonumber\\
& + \sum_{s \in x_i^{7m}[x_i^3] } \i \sqrt{3}( \chi_j(s)-\chi_j(s^{-1})) \nonumber\\
=& -2\sqrt{3} \Im(\chi_j(x_i^{m}))\sum_{s \in [x_i^3]} \chi_j(s)  -2\sqrt{3} \Im(\chi_j(x_i^{4m}))\sum_{s \in [x_i^3]} \chi_j(s)  \nonumber\\
&- 2\sqrt{3} \Im(\chi_j(x_i^{7m}))\sum_{s \in [x_i^3]} \chi_j(s) \nonumber\\
=& -2\sqrt{3}\bigg(\sum\limits_{r\in G_{9,3}^1(1)} \Im(\chi_j(x_i^{rm})) \bigg) \sum_{s \in [x_i^3]} \chi_j(s) \label{NewEqFinalCorr1}
\end{align}
for each $i \in \{ 1, \ldots, k\}$. Similarly, for $m \equiv 2 \Mod 3$ we have
\begin{align}
2\sum_{s \in \langle\!\langle x_i \rangle\!\rangle} \i \sqrt{3}( \chi_j(s)&-\chi_j(s^{-1})) = -2\sqrt{3}\bigg(\sum\limits_{r\in G_{9,3}^2(1)} \Im(\chi_j(x_i^{rm})) \bigg)\sum_{s \in [x_i^3]} \chi_j(s)\label{NewEqFinalCorr2}
\end{align}
for each $i \in \{ 1, \ldots, k\}$. 
Using Equation (\ref{NewEqFinalCorr1}) and Equation (\ref{NewEqFinalCorr2}), we get
\begin{align}
2T_x(j) &= \frac{1}{\chi_j(\mathbf 1)} \sum_{i=1}^{k} 2 \sum_{s \in \langle\!\langle x_i \rangle\!\rangle} \i \sqrt{3}( \chi_j(s)-\chi_j(s^{-1})) \nonumber\\
&= \left\{ \begin{array}{ll}
					 -2\sqrt{3}\bigg(\sum\limits_{r\in G_{9,3}^1(1)} \Im(\chi_j(x^{rm})) \bigg) \frac{1}{\chi_j(\mathbf 1)}\sum\limits_{i=1}^{k}  \sum\limits_{s \in [ x_i^3 ]} \chi_j(s) & \mbox{if } m \equiv 1 \Mod 3 \\
					 -2\sqrt{3}\bigg(\sum\limits_{r\in G_{9,3}^2(1)} \Im(\chi_j(x^{rm})) \bigg)\frac{1}{\chi_j(\mathbf 1)}\sum\limits_{i=1}^{k}  \sum\limits_{s \in [ x_i^{3} ]}\chi_j(s) & \mbox{if } m \equiv 2 \Mod 3
				\end{array}\right.\nonumber\\
&= \left\{ \begin{array}{ll}
					-2\sqrt{3}\bigg(\sum\limits_{r\in G_{9,3}^1(1)} \Im(\chi_j(x^{rm})) \bigg) C_{x^3}(j) & \mbox{if } m \equiv 1 \Mod 3 \\
					-2\sqrt{3}\bigg(\sum\limits_{r\in G_{9,3}^2(1)} \Im(\chi_j(x^{rm})) \bigg) C_{x^3}(j)  & \mbox{if } m \equiv 2 \Mod 3.
				\end{array}\right.\label{EqNewSum4mLemma77}
\end{align}
The last equality in the preceding equations follows from Lemma~\ref{NewLemmaEquivaTrans3Ch5}. 

Let $d_j=\chi_j(\mathbf 1)$. Assume that $t=2$. By Theorem~\ref{NewThmChap1Added}, we have $\chi_j(x^{m}) = \sum\limits_{\ell=1}^{d_j} \epsilon_{j\ell}$, where $\epsilon_{j1},\ldots ,\epsilon_{jd_j}$ are some $9$-th roots of unity. We have
\begin{equation} \label{EqNewSum4mLemma66}
\begin{split}
-2\sqrt{3} \sum\limits_{r\in G_{9,3}^1(1)} \Im(\chi_j(x^{rm})) &= \i \sqrt{3} \sum\limits_{r\in G_{9,3}^1(1)} (\chi_j(x^{rm}) -\chi_j(x^{-rm}))\\
&=\i \sqrt{3}\sum\limits_{r\in G_{9,3}^1(1)} \bigg( \sum\limits_{\ell=1}^{d_j} \epsilon_{j\ell}^r - \sum\limits_{\ell=1}^{d_j} \epsilon_{j\ell}^{-r} \bigg)\\
&= \sum\limits_{\ell=1}^{d_j} \sum\limits_{r\in G_{9,3}^1(1)} \i \sqrt{3} (\epsilon_{j\ell}^r - \epsilon_{j\ell}^{-r}).
\end{split} 
\end{equation} 

Note that $\sum\limits_{r\in G_{9,3}^1(1)}\i \sqrt{3} (\epsilon_{j\ell}^r - \epsilon_{j\ell}^{-r})= \i \sqrt{3}(\epsilon_{j\ell} - \epsilon_{j\ell}^2)(1+\epsilon_{j\ell}^3 + \epsilon_{j\ell}^6)$. Since $\epsilon_{j\ell} \in \{1,\omega_9,\omega_9^2,\ldots , \omega_9^8\}$, we see that 
\[ \sum\limits_{r\in G_{9,3}^1(1)}\i \sqrt{3} (\epsilon_{j\ell}^r - \epsilon_{j\ell}^{-r}) = \left\{ \begin{array}{ll}
					 \pm 9 & \mbox{if } \epsilon_{j\ell} \in \{\omega_9^3,\omega_9^6\} \\
					 0 & \mbox{otherwise}.
				\end{array}\right.\]
Thus $\sum\limits_{r\in G_{9,3}^1(1)}\i \sqrt{3} (\epsilon_{j\ell}^r - \epsilon_{j\ell}^{-r})$ is an integer multiple of $3$. Therefore by Equation~(\ref{EqNewSum4mLemma66}), we find that $-2\sqrt{3} \sum\limits_{r\in G_{9,3}^1(1)} \Im(\chi_j(x^{rm}))$ is an integer multiple of $3$. Similarly, $-2\sqrt{3} \sum\limits_{r\in G_{9,3}^2(1)} \Im(\chi_j(x^{rm}))$ is also an integer multiple of $3$. Using Equation~(\ref{EqNewSum4mLemma77}), we find that $\frac{2T_x(j)}{3}$ is an integer. Since $T_x(j)$ is an integer, integrality of $\frac{2T_x(j)}{3}$ gives that $\frac{T_x(j)}{3}$ is also an integer for each $j \in \{ 1, \ldots, h\}$. 

 Now assume that $t \geq 3$ and $j \in \{ 1,\ldots h \}$. Let 
$$A_x(j):= \left\{ \begin{array}{ll}
					-2\sqrt{3}\bigg(\sum\limits_{r\in G_{9,3}^1(1)} \Im(\chi_j(x^{rm})) \bigg)  & \mbox{if } m \equiv 1 \Mod 3 \\
					-2\sqrt{3}\bigg(\sum\limits_{r\in G_{9,3}^2(1)} \Im(\chi_j(x^{rm})) \bigg)  & \mbox{if } m \equiv 2 \Mod 3.
				\end{array}\right.$$
By Equation~(\ref{EqNewSum4mLemma77}), we find that $2T_x(j)= A_x(j) C_{x^3}(j)$. Therefore $A_x(j)$ is a rational algebraic integer whenever $C_{x^3}(j)\neq 0$. Thus, if $C_{x^3}(j)\neq 0$ then $A_x(j)$ is an integer. Now by Lemma~\ref{NewSum4mLemma11Ch5} and Equation~(\ref{EqNewSum4mLemma77}), $\frac{2T_x(j)}{3}$ is an integer, and hence $\frac{T_x(j)}{3}$ is also an integer.
\end{proof}

Let $S$ be a nonempty set in $\mathbb{E}(\Gamma)$ and $S$ be expressible as a union of some conjugacy classes of $\Gamma$. Then $S$ is a skew-symmetric subset of $\Gamma$ that is closed under both conjugacy and the equivalence relation $\simeq$. Let $S={\rm Cl}(x_1)\cup\cdots \cup  {\rm Cl}(x_k)=\langle\!\langle  y_1 \rangle\!\rangle  \cup \cdots \cup \langle\!\langle  y_r \rangle\!\rangle$ for some $x_1,\ldots,x_k,y_1,\ldots,y_r\in \Gamma(3)$. We see that
\begin{align*}
S=  {\rm Cl}(x_1)\cup\cdots \cup  {\rm Cl}(x_k)
= \left(\text{\footnotesize$\bigcup\limits_{s\in {\rm Cl}(x_1)}$} \langle\!\langle  s \rangle\!\rangle  \right)\cup \cdots \cup  \left(\text{\footnotesize$\bigcup\limits_{s\in {\rm Cl}(x_k)}$} \langle\!\langle  s \rangle\!\rangle  \right) 
= S_{x_1}^3 \cup\cdots \cup S_{x_k}^3 .
\end{align*}
Due to Lemma~\ref{partition3}, we can assume that the sets $S_{x_1}^3 ,\ldots , S_{x_k}^3$ are all distinct. In the following result, we also prove the converse of Theorem~\ref{MinCharacEisensteinIntegCh5}.

\begin{theorem}\label{MinCharacEisensteinInteg2Ch5}
If $\Gamma$ is a finite group, then the normal mixed Cayley graph $\Cay(\Gamma,S)$ is Eisenstein integral if and only if it is HS-integral.
\end{theorem}
\begin{proof}
Assume that $\Cay(\Gamma,S)$ is HS-integral and $j \in \{ 1,\ldots , h \}$. Then $\Cay(\Gamma,S\setminus \overline{S})$ is integral, and so $f_{j}(S)$ is an integer. By Theorem~\ref{MixedHSintegralChara}, $\overline{S} \in \mathbb{E}(\Gamma)$, which implies that $\overline{S}= S_{x_1}^3 \cup  \cdots \cup S_{x_k}^3 $ for some $x_1,\ldots,x_k \in \Gamma (3)$, where the sets $S_{x_1}^3, \ldots , S_{x_k}^3$ are all distinct. Using the fact that $S_{x_i}^3 \cup S_{x_i^{-1}}^3 = S_{x_i}^1$, we have $\overline{S} \cup \overline{S}^{-1}= S_{x_1}^1 \cup  \cdots \cup S_{x_k}^1 $. Therefore
\begin{align}
g_{j}(S)&=\frac{1}{2\chi_{j}(\mathbf{1})} \sum\limits_{s\in \overline{S}\cup \overline{S}^{-1}}\chi_{j}(s) -\frac{1}{6\chi_{j}(\mathbf{1})} \sum\limits_{s\in \overline{S}}\i \sqrt{3}\left(\chi_{j}(s)-\chi_{j}(s^{-1})\right) \nonumber\\
&=\frac{1}{2\chi_{j}(\mathbf{1})} \sum\limits_{\ell=1}^{k} \sum\limits_{s\in S_{x_{\ell}}^1 }\chi_{j}(s) -\frac{1}{6\chi_{j}(\mathbf{1})} \sum\limits_{\ell=1}^{k} \sum\limits_{s\in S_{x_{\ell}}^3 }\i \sqrt{3}\left(\chi_{j}(s)-\chi_{j}(s^{-1})\right)\nonumber\\
&= \frac{1}{2} \sum\limits_{\ell=1}^{k} C_{x_{\ell}}(j) - \frac{1}{6} \sum\limits_{\ell=1}^{k} T_{x_{\ell}}(j)\nonumber\\
&= \frac{1}{2} \sum\limits_{\ell=1}^{k} \left( C_{x_{\ell}}(j) - \frac{1}{3} T_{x_{\ell}}(j) \right).\label{eqNormEisen}
\end{align}

Let $1 \leq \ell \leq k$. Since $\frac{C_{x_{\ell}}(j)+T_{x_{\ell}}(j)}{2}$ is an HS-eigenvalue of the normal oriented Cayley graph $\Cay(\Gamma, S^3_{x_{\ell}})$, the numbers $C_{x_{\ell}}(j)$ and $T_{x_{\ell}}(j)$ are integers of the same parity. By Lemma~\ref{NewSum4mLemma22} and Lemma~\ref{NewSum4mLemma33Ch5}, $\frac{T_{x_{\ell}}(j)}{3}$ is an integer. Therefore, $C_{x_{\ell}}(j)$ and $\frac{T_{x_{\ell}}(j)}{3}$ are integers of the same parity. Thus $C_{x_{\ell}}(j) - \frac{1}{3} T_{x_{\ell}}(j)$ is an even integer, and so $g_j(S)$ is an integer by Equation (\ref{eqNormEisen}). Hence by Lemma~\ref{CharaEisensteinIntegral1Ch5}, $\Cay(\Gamma,S)$ is Eisenstein integral. The other part of the theorem is proved in Theorem~\ref{MinCharacEisensteinIntegCh5}. 
\end{proof}

The following example illustrates an use of  Theorem~\ref{MinCharacEisensteinInteg2Ch5}.

\begin{ex} \normalfont
Consider the  mixed graph $\Cay(A_4, S)$ of Example \ref{exampleMixed}. We have already seen that it is HS-integral, and hence it must be Eisenstein integral. We find that the spectrum of $\Cay(A_4, S)$ is $\{[\gamma_1]^1, [\gamma_2]^1, [\gamma_3]^1, [\gamma_4]^9 \}$, where $\gamma_1=7,\gamma_2= 3+4\omega_3 , \gamma_3 = -1-4\omega_3,$ and $\gamma_4=-1 .$ It is clear that the eigenvalues of $\Cay(A_4, S)$ are Eisenstein integers.
\end{ex}


\noindent\textbf{Acknowledgements}

I sincerely thank Dr. Bikash Bhattacharjya for his guidance, enthusiastic encouragement and useful critiques of this research work.



\begin{thebibliography}{10}


\bibitem{ahmadi2009graphs}
O.~Ahmadi, N.~Alon, I.F. Blake and I.E. Shparlinski.
\newblock Graphs with integral spectrum.
\newblock {\em Linear Algebra and its Applications}, 430(1):547--552, 2009.

\bibitem{alperin2012integral}
R.C. Alperin and B.L. Peterson.
\newblock Integral sets and Cayley graphs of finite groups.
\newblock {\em The Electronic Journal of Combinatorics}, 19(1): \#P44, 2012.

\bibitem{balinska2002survey}
K.~Bali{\'n}ska, D.~Cvetkovi{\'c}, Z.~Radosavljevi{\'c}, S.~Simi{\'c} and
  D.~Stevanovi{\'c}.
\newblock A survey on integral graphs.
\newblock {\em Publikacije Elektrotehni{\v{c}}kog fakulteta. Serija Matematika,
  $42\mbox{--}65$}, 2002.

\bibitem{bridges1982rational}
W.G. Bridges and R.A. Mena.
\newblock Rational g-matrices with rational eigenvalues.
\newblock {\em Journal of Combinatorial Theory, Series A}, 32(2):264--280,
  1982.




\bibitem{cheng2019integral}
T.~Cheng, L.~Feng and H.~Huang.
\newblock Integral Cayley graphs over dicyclic group.
\newblock {\em Linear Algebra and its Applications}, 566:121--137, 2019.

\bibitem{csikvari2010integral}
P.~Csikv{\'a}ri.
\newblock Integral trees of arbitrarily large diameters.
\newblock {\em Journal of Algebraic Combinatorics}, 32(3):371--377, 2010.

\bibitem{dummitandfoote}
D.S.~Dummit and R.M.~Foote.
\newblock {\em Abstract Algebra, Third Edition}. Hoboken: Wiley, 2004.

\bibitem{foster2016spectra}
B.~Foster-Greenwood and C.~Kriloff.
\newblock Spectra of Cayley graphs of complex reflection groups.
\newblock {\em Journal of Algebraic Combinatorics}, 44(1):33--57, 2016.

\bibitem{godsil2014rationality}
C.~Godsil and P.~Spiga.
\newblock Rationality conditions for the eigenvalues of normal finite Cayley
  graphs.
\newblock { arXiv preprint arXiv:1402.5494}, 2014.


\bibitem{harary1974graphs}
F.~Harary and A.J. Schwenk.
\newblock Which graphs have integral spectra?
\newblock {\em Graphs and Combinatorics, $45\mbox{--}51$}, 1974.

\bibitem{james2001representations}
G.~James and M.~Liebeck.
\newblock {\em Representations and characters of groups}.
\newblock Cambridge University Press, 2001.


\bibitem{kadyanArxivNorm}
M.~Kadyan and B.~Bhattacharjya.
\newblock H-integral and Gaussian integral normal mixed Cayley graphs.
arXiv:2110.03268.




\bibitem{kadyan2021hsIntegralAbelian}
M.~Kadyan and B.~Bhattacharjya.
\newblock HS-integral and Eisenstein integral mixed Cayley graphs over abelian groups.
\newblock {\em Linear Algebra and its Applications}, 645: 68-90, 2022.

\bibitem{kadyan2021Secintegral}
M.~Kadyan and B.~Bhattacharjya.
\newblock HS-integral and Eisenstein integral mixed circulant graphs.
\textit{Theory and Applications of Graphs}, 10(1):3, 2023.





\bibitem{klotz2010integral}
W.~Klotz and T.~Sander.
\newblock Integral Cayley graphs over abelian groups.
\newblock {\em The Electronic Journal of Combinatorics}, 17:\#R81, 2010.

\bibitem{ku2015cayley}
C.Y. Ku, T.~Lau and K.B. Wong.
\newblock Cayley graph on symmetric group generated by elements fixing k
  points.
\newblock {\em Linear Algebra and its Applications}, 471:405--426, 2015.


\bibitem{li2022hermitian}
S.~Li and Y.~Yu.
\newblock Hermitian adjacency matrix of the second kind for mixed graphs.
\newblock {\em Discrete Mathematics}, 345(5):112798, 2022.

\bibitem{2015mixed}
J.~Liu and X.~Li.
\newblock Hermitian-adjacency matrices and Hermitian energies of mixed graphs.
\newblock {\em Linear Algebra and its Applications}, 466:182--207, 2015.

\bibitem{lu2018integral}
L.~Lu, Q.~Huang and X.~Huang.
\newblock Integral Cayley graphs over dihedral groups.
\newblock {\em Journal of Algebraic Combinatorics}, 47(4):585--601, 2018.

\bibitem{mohar2020new}
B.~Mohar.
\newblock A new kind of Hermitian matrices for digraphs.
\newblock {\em Linear Algebra and its Applications}, 584:343--352, 2020.

\bibitem{2006integral}
W.~So.
\newblock Integral circulant graphs.
\newblock {\em Discrete Mathematics}, 306(1):153--158, 2006.

\bibitem{steinberg2009representation}
B.~Steinberg.
\newblock {\em Representation theory of finite groups}.
\newblock Springer New York, 2009.


%

\bibitem{watanabe1979note}
M.~Watanabe.
\newblock Note on integral trees.
\newblock {\em Mathematics Reports}, 2:95--100, 1979.

\bibitem{watanabe1979integral}
M.~Watanabe and A.J. Schwenk.
\newblock Integral starlike trees.
\newblock {\em Journal of the Australian Mathematical Society}, 28(1):120--128,
  1979.

\end{thebibliography}
\end{document}